\documentclass[11pt,a4paper]{siamart171218}
\usepackage[utf8]{inputenc}
\usepackage[english]{babel}

\usepackage{tikz}
\usetikzlibrary{shapes.geometric}

\usepackage{graphicx}
\usepackage{fancyhdr}
\usepackage{dsfont}
\usepackage{stmaryrd}

\usepackage{amsfonts}
\usepackage{amssymb}

\usepackage{amsmath}

\usepackage{amsfonts}
\usepackage{amssymb}
\usepackage{mathtools}
\usepackage{enumitem}

\usepackage{mathrsfs}

\usepackage{overpic} 

\usepackage{algorithm}
\usepackage[noend]{algpseudocode}

\usepackage{mathtools,booktabs}

\DeclarePairedDelimiter\floor{\lfloor}{\rfloor}

\usepackage{varwidth}

\usepackage[title]{appendix}

\newcommand{\ttrank}{{\rm rank}^{\rm TT}}
\newcommand{\mlrank}{{\rm rank}^{\rm ML}}

\newcommand{\ttsto}{p^{{\rm TT}}}
\newcommand{\mlsto}{p^{{\rm ML}}}
\newcommand{\cpsto}{p^{{\rm CP}}}
\newcommand{\rank}{{\rm rank}}

\newcommand{\red}{\textcolor{black}}

\usepackage{cite}
\usepackage{color}

\title{On the compressibility of tensors\thanks{Submitted to the editors \today.
\funding{This work is supported by National Science Foundation grant no.~1818757.}}}
\author{Tianyi Shi\thanks{Center for Applied Mathematics, Cornell University, Ithaca, NY 14853. (\email{ts777@cornell.edu})} \and Alex Townsend\thanks{Department of Mathematics, Cornell University, Ithaca, NY  14853. (\email{townsend@cornell.edu})}}
\headers{Numerical tensor ranks}{Tianyi Shi and Alex Townsend}

\begin{document}
\newcommand{\R}[0]{\mathbb{R}}
\newcommand{\C}[0]{\mathbb{C}}
\maketitle

\begin{abstract}
Tensors are often compressed by expressing them in low rank tensor formats.  In this paper, we develop three methodologies that bound the compressibility of a tensor: (1) Algebraic structure, (2) Smoothness, and (3) Displacement structure. For each methodology, we derive bounds on \red{storage costs} that partially explain the abundance of compressible tensors in applied mathematics. For example, we show that the solution tensor $\mathcal{X} \in \C^{n \times n \times n}$ of a discretized Poisson equation $-\nabla^2 u =1$ on $[-1,1]^3$ with zero Dirichlet conditions can be approximated to a relative accuracy of $0<\epsilon<1$ \red{in the Frobenius norm by a tensor in tensor-train format with $\mathcal{O}(n (\log n)^2 (\log(1/\epsilon))^2)$ degrees of freedom. As this bound is constructive, we are also able to solve this equation spectrally with $\mathcal{O}(n (\log n)^3 (\log(1/\epsilon))^3)$ complexity}.
\end{abstract}



\section{Introduction}\label{sec:introduction}
A wide variety of applications, \red{such as approximation theory~\cite{khoromskij2011tensor}, continuum mechanics~\cite{eshelby1999energy}, differential equations~\cite{khoromskij2011numerical,kressner2011preconditioned}, and data analysis~\cite{lu2011survey}}, lead to problems involving data or solutions that can be represented by tensors~\cite{kolda2009tensor}. A general $d$-order tensor $\mathcal{X} \in\mathbb{C}^{n_1\times\cdots \times n_d}$ has $\prod_{k=1}^d n_k$ entries, which prevents it from being stored explicitly except for modest $d$.  It is often essential to represent or approximate tensors using sparse data formats, such as low rank tensor decompositions~\cite{grasedyck2013literature,kolda2009tensor}. However, the need for data sparse formats does not imply that such approximations are always mathematically possible. In this paper, we derive bounds on numerical tensor ranks for certain families of tensors, and in doing so, we partially justify the use of low rank tensor decompositions. Analogous theoretical results have already been derived that explicitly bound the numerical rank of matrices~\cite{beckermann2017singular,massei2018solving,reade1983eigenvalues,townsend2017singular}.

The situation for tensors is more complicated than for matrices, and this is reflected in several distinct low rank tensor decompositions~\cite{kolda2009tensor,oseledets2011tensor}.  Here, we consider three such decompositions:  (a) Tensor-train decomposition (see~\cref{sec:TT}), (b) Orthogonal Tucker decomposition (see~\cref{sec:OrthogonalTucker}), and (c) Canonical Polyadic (CP) decomposition (see~\cref{sec:CPdecomposition}). These three tensor decompositions supply three different definitions of tensor rank, and therefore each one requires separate attention.  

For a given tensor $\mathcal{X} \in\mathbb{C}^{n_1\times\cdots \times n_d}$, we are interested in developing a variety of tools to theoretically explain whether there exists a low rank tensor $\tilde{\mathcal{X}}\in\mathbb{C}^{n_1\times\cdots \times n_d}$, in one or more of the tensor formats, such that
\begin{equation} 
\| \mathcal{X} - \tilde{\mathcal{X}} \|_F \leq \epsilon \| \mathcal{X} \|_F, \qquad \|\mathcal{X}\|_F^2 = \sum_{i_1=1}^{n_1} \cdots \sum_{i_d = 1}^{n_d} |\mathcal{X}_{i_1,\ldots,i_d}|^2,
\label{eq:FrobeniusNorm}
\end{equation} 
where $0\leq\epsilon<1$ is an accuracy tolerance. If $\mathcal{X}$ can be well-approximated by $\tilde{\mathcal{X}}$, then dramatic storage and computational benefits can be achieved by replacing $\mathcal{X}$ by $\tilde{\mathcal{X}}$~\cite{grasedyck2013literature,hackbusch2012tensor}.  \red{We say that a tensor is compressible if it can be approximated by a low rank tensor, in the sense of~\eqref{eq:FrobeniusNorm} that can be represented in a relative small number of degrees of freedom. In order to compare different low rank tensor formats, we examine the number of degrees of freedom required to store an approximate tensor.}

In this paper, we explore three methodologies to bound the \red{compressibility} of a tensor:
\begin{itemize}[leftmargin=*,noitemsep] 
\item \textbf{Algebraic structures:} If a tensor is constructed by sampling a multivariable function that can be expressed as a sum of products of single-variable functions, then that tensor \red{is often compressible}. Occasionally, one may have to perform algebraic manipulations to a function to explicitly reveal its \red{desired} structure, for example, by using trigonometric identities (see~\cref{sec:algebraic}).

\item \textbf{Smoothness:} If a tensor can be constructed by sampling a smooth function on a tensor-product grid, \red{then that tensor is often compressible}. This observation can be made rigorous by using the fact that smooth functions on compact domains can be well-approximated by polynomials~\cite{trefethen2013approximation, hackbusch2012tensor}. 
  
\item \textbf{Displacement structure:} If a tensor $\mathcal{X}$ satisfies a multidimensional Sylvester equation of the form:  
\begin{equation} 
\mathcal{X} \times_1 A^{(1)} + \cdots + \mathcal{X} \times_d A^{(d)} = \mathcal{G}, \qquad A^{(k)} \in \mathbb{C}^{n_k\times n_k}, \quad \mathcal{G} \in\mathbb{C}^{n_1\times \cdots \times n_d},
\label{eq:TensorDisplacement} 
\end{equation} 
where `$\times_k$' denotes the $k$-mode matrix product of a tensor (see~\cref{eq:kfold}), then this---under additional assumptions---can ensure that the tensor $\mathcal{X}$ is well-approximated by a low rank tensor. Multidimensional Sylvester equations such as~\cref{eq:TensorDisplacement} appear when discretizing certain partial differential equations with finite differences~\cite{leveque2007finite} and are satisfied by several classes of structured tensors~\cite{grigorascu1999tensor}. For example, we show that the solution tensor $\mathcal{X} \in \C^{n \times n \times n}$ to $-\nabla^2u=1$ on $[-1,1]^3$ can be represented up to a relative accuracy of $0<\epsilon<1$ \red{in the Frobenius norm} with just $\mathcal{O}(n(\log n)^2 (\log(1/\epsilon))^2)$ degrees of freedom in tensor-train format, despite the solution having weak corner singularities.
\end{itemize} 

\red{The first two methodologies are considered in~\cite{hackbusch2012tensor}, and the third methodology is related to the literature on exponential sums~\cite{grasedyck2004existence,khoromskij2009tensor}. In this manuscript, we formally provide bounds on the compressibility of such tensors and illustrate the methodologies with worked examples. } 

After some experience, one can successfully identify which methodology is likely to result in the best theoretical bounds on the compressibility of a tensor. We emphasize that these three methodologies provide upper bounds \red{using} numerical tensor ranks, and do not provide a complete characterization on the \red{compressibility} of tensors. Another approach that partially explains the abundance of tensors with small storage is artificial coordinate alignment~\cite{trefethen2017cubature}, though we do not know how to use this observation to derive explicit bounds on tensor ranks. 

\subsection{Tensor notation} 
We follow as closely as possible the notation for tensors found in~\cite{kolda2009tensor}, which we briefly review now for the reader's convenience.

\begin{description}[leftmargin=*,noitemsep]
\item[The $k$-mode product.] The $k$-fold (or $k$-mode) product of a tensor $\mathcal{X}\in\C^{n_1\times\cdots\times n_d}$ with a matrix $A\in\mathbb{C}^{n_k\times n_k}$ is denoted by $\mathcal{X} \times_k A$, and defined elementwise as
\begin{equation}
(\mathcal{X} \times_k A)_{i_1,\ldots,i_{k-1},j,i_{k+1},\ldots,i_d} = \sum_{i_k = 1}^{n_k} \mathcal{X}_{i_1,\ldots,i_d}A_{j,i_k}.
\label{eq:kfold} 
\end{equation} 
It corresponds to each mode-$k$ fiber of $X$ being multiplied by the matrix $A$. 

\item[Double bracket.]
In the tensor literature, the double bracket \red{denotes a mapping from the parametric space to the space of tensors. Specifically, it can be considered as} a weighted sums of rank-1 tensors, i.e., 
\begin{equation} 
\llbracket \mathcal{G} ; A^{(1)},\ldots,A^{(d)} \rrbracket =  \sum_{i_1=1}^{r_1} \cdots \sum_{i_d=1}^{r_d} \mathcal{G}_{i_1,\ldots,i_d} A^{(1)}_{i_1} \circ \cdots \circ A^{(d)}_{i_d}, \qquad A^{(k)} \in \mathbb{C}^{n_k\times r_k}, 
\label{eq:bracket} 
\end{equation} 
where $\mathcal{G} \in \mathbb{C}^{r_1\times \cdots \times r_d}$ is often referred to as the core tensor and $v_1\circ \cdots \circ v_d$ is the $d$-way outer-product of vectors~\cite{kolda2009tensor}.

\item[Flattening by reshaping.]
One can always reorganize the entries of a tensor into a matrix and this idea is fundamental to the tensor-train decomposition~\cite{oseledets2011tensor}. Conventionally, one reorganizes the entries so that the mode-1 fibers are stacked. This is equivalent to $X_k={\rm reshape}(\mathcal{X},\prod_{s=1}^k n_s,\prod_{s=k+1}^d n_s)$.\footnote{In MATLAB, the reshape command reorganizes the entries of a tensor. For example, if $\mathcal{X} \in \C^{n_1 \times \dots \times n_d}$, then ${\rm reshape}(\mathcal{X},\prod_{s=1}^k n_s,\prod_{s=k+1}^d n_s)$ returns a matrix of size $(\prod_{s=1}^k n_s) \times (\prod_{s=k+1}^d n_s)$ formed by stacking entries according to their multi-index.} We call $X_k$ the $k$th unfolding of $\mathcal{X}$.

\item[Flattening via matricization.]
Another way to flatten a tensor is called mode-$n$ matricization (or $n$th matricization), which arranges the mode-$n$ fibers to be the columns of a matrix~\cite{kolda2006multilinear}. We denote the mode-$n$ matricization of a tensor $\mathcal{X}$ by $X_{(n)}$. It is easy to see that the first unfolding and the mode-1 matricization of a tensor are identical, i.e., $X_{(1)}=X_1$. In this paper, for a tensor $\mathcal{X}$, matricizations are constructed so that there exists another tensor $\mathcal{Y}^j$ satisfying~\cite{de2000multilinear} 
\begin{equation} \label{eq:cyc_mat}
Y^j_{(1)}=X_{(j)}, \ \dots, \ Y^j_{(d-j+1)}=X_{(d)}, \ Y^j_{(d-j+2)}=X_{(1)}, \ \dots,\  Y^j_{(d)}=X_{(j-1)}.
\end{equation}
\end{description} 

\subsection{Summary of paper}
In the next section, we review three tensor decompositions, and in~\cref{sec:algebraic} we study the ranks of tensors that are constructed by sampling multivariate functions that have some algebraic structure. In~\cref{sec:smoothness}, we consider the \red{storage cost} of tensors constructed by sampling smooth multivariate functions. Finally, in~\cref{sec:displacement} we consider tensors that satisfy a multidimensional Sylvester equation\red{, including a fast tensor Sylvester equation solver that exploits the compressibility of these tensors~\cref{sec:poisson_solver}}.

\section{Three tensor decompositions}\label{sec:tensorDecompositions} 
In this section, we review three tensor decompositions: (a) Tensor-train decomposition, (b) Orthogonal Tucker decomposition, and (c) CP decomposition. \red{For each decomposition, and a given $0 < \epsilon < 1$, we say a tensor $\mathcal{X}$ is $p$-compressible if there exists a tensor $\mathcal{\tilde{X}}$ that can be represented with $p$ degrees of freedom and $||\mathcal{X}-\mathcal{\tilde{X}}||_F \le \epsilon ||\mathcal{X}||_F$.}

\subsection{Tensor-train decomposition}\label{sec:TT} 
The tensor-train decomposition is a generalization of the singular value decomposition (SVD) that can be computed by a sequence of matrix SVDs~\cite{oseledets2011tensor, oseledets2009breaking}.  A tensor $\mathcal{X}\in\mathbb{C}^{n_1\times \cdots \times n_d}$ has a tensor-train rank of at most $\pmb{s}=(s_0,\ldots,s_d)$, if there exists matrix-valued functions $G_k : \{1,\ldots,n_k\} \mapsto \mathbb{C}^{s_{k-1}\times s_k}$ for $1\leq k\leq d$ such that 
\begin{equation}
\mathcal{X}_{i_1,\ldots,i_d} = G_1(i_1)G_2(i_2) \cdots G_d(i_d), \qquad 1\leq i_k \leq n_k.
\label{eq:TensorTrain} 
\end{equation} 
This decomposition writes each entry of $\mathcal{X}$ as a product of matrices, where the $k$th matrix in the ``train" of length $d$ is determined by $i_k$. Since the product of the matrices must always be a scalar, we have $s_0 = s_d = 1$.  Each $G_k$ can be represented by an $s_{k-1}\times n_k \times s_k$ tensor so a tensor-train decomposition of rank at most $\pmb{s}$ requires $\ttsto \le \sum_{k=1}^d s_{k-1}s_k n_k$ \red{degrees of freedom to store the format in memory}. \Cref{fig:TT} illustrates a tensor-train decomposition of rank at most $\pmb{s} = (s_0,\ldots,s_d)$.
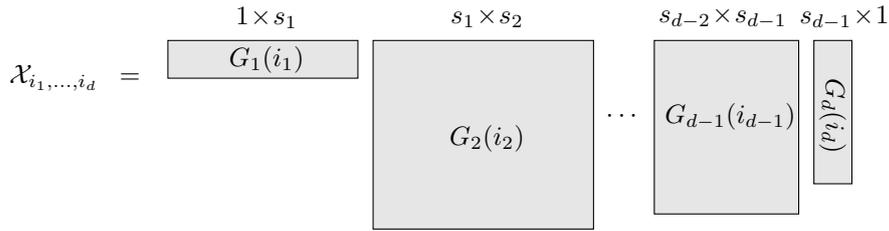
\begin{figure} 
\centering 
\begin{tikzpicture}
\filldraw[black] (0,-0.5)  node {$\mathcal{X}_{i_1,\ldots,i_d}$};
\filldraw[black] (1,-0.5) node {$=$};
\filldraw[color=black,fill=gray!20] (1.5,0) rectangle (4,-.5);
\filldraw[black] (2.8,-0.25) node {$G_1(i_1)$};
\filldraw[black] (2.8,0.3) node {$1\! \times \! s_1$};
\filldraw[color=black,fill=gray!20] (4.2,0) rectangle (7.1,-2.5);
\filldraw[black] (5.7,-1.3) node {$G_2(i_2)$};
\filldraw[black] (5.7,0.3) node {$s_1 \! \times \! s_2$};
\filldraw[black] (7.5,-1) node {$\cdots$};
\filldraw[color=black,fill=gray!20] (7.9,0) rectangle (9.8,-2.3);
 \filldraw[black] (8.9,-1) node {$G_{d-1}(i_{d-1})$};
\filldraw[black] (8.8,0.3) node {$s_{d-2} \! \times \! s_{d-1}$};
\filldraw[color=black,fill=gray!20] (10,0) rectangle (10.5,-1.9);
\filldraw[black] (10.25,-1) node {\rotatebox{270}{$G_{d}(i_{d})$}};
\filldraw[black] (10.4,0.3) node {$s_{d-1} \! \times \! 1$};
\end{tikzpicture}
\caption{The tensor-train decomposition of rank at most $\pmb{s} = (s_0,\ldots,s_d)$. Each entry of a tensor is represented by the product of $d$ matrices, where the $k$th matrix in the ``train" is selected based on the value of $i_k$.}
\label{fig:TT}
\end{figure} 

Normally a tensor-train decomposition is constructed by separating out one dimension at a time, and compressing each dimension in turn~\cite{oseledets2011tensor}. For simplicity, the decomposition considered in this paper is performed in the order of dimension 1, dimension 2, and so on. In this way, the entries of the tensor-train rank are bounded from above by the ranks of matrices formed by flattening~\cite{oseledets2011tensor}. That is, for $1 \le k \le d-1$ we have
\begin{equation} \label{eq:TT_trivial}
s_k \leq {\rm rank}(X_k), \qquad X_k={\rm reshape}(\mathcal{X},\prod_{s=1}^k n_s,\prod_{s=k+1}^d n_s),
\end{equation}
where $\rank(X_k)$ is the rank of the matrix $X_k$.  Therefore, if the ranks of all the matrices $X_k$ for $1\leq k\leq d-1$ are small, then the tensor $X$ can be exactly represented in a data-sparse format as a tensor-train decomposition. 


\subsection{Orthogonal Tucker decomposition}\label{sec:OrthogonalTucker}
The orthogonal Tucker decomposition is a factorization of a tensor into a set of matrices and a core tensor, where the matrices have orthonormal columns~\cite{hitchcock1927expression,kolda2009tensor,de2000multilinear}. A tensor $\mathcal{X} \in \mathbb{C}^{n_1\times\cdots \times n_d}$ has a multilinear rank\footnote{The closely related Tucker rank of a tensor is also associated to the Tucker decomposition, except the matrices $A_k$ in~\cref{eq:Tucker} are not constrained to have orthonormal columns. Since multilinear ranks are more commonly used in applications, we do not consider Tucker ranks in this paper.} of at most $\pmb{t}=(t_1,\ldots,t_d)$, if there are matrices $A_1,\ldots,A_d$ with orthonormal columns and a core tensor $\mathcal{G} \in\mathbb{C}^{t_1 \times \cdots \times t_d}$ such that 
\begin{equation}
\mathcal{X} = \llbracket \mathcal{G}; A^{(1)}, \ldots, A^{(d)} \rrbracket,  \qquad A^{(k)} \in \mathbb{C}^{n_k\times t_k}.
\label{eq:Tucker} 
\end{equation}
Such a decomposition contains $\red{\mlsto \le} \sum_{k=1}^{d} n_k t_k + \prod_{k=1}^{d} t_k$ degrees of freedom, and can be computed by the so-called higher-order singular value decomposition~\cite{de2000multilinear}. 


\subsection{Canonical Polyadic decomposition}\label{sec:CPdecomposition} 
The CP decomposition expresses a tensor as a sum of rank-1 tensors. A tensor $\mathcal{X} \in \mathbb{C}^{n_1\times\cdots \times n_d}$ is of rank at most $r$, if there are matrices $A^{(1)},\ldots,A^{(d)}$ and a diagonal tensor $D$ that
\begin{equation} 
\mathcal{X} = \llbracket \mathcal{D} ; A^{(1)}, \ldots, A^{(d)} \rrbracket, \qquad A^{(k)}\in\mathbb{C}^{n_k\times r}, \quad \mathcal{D}\in\mathbb{C}^{r\times \cdots \times r},
\label{eq:CPdecomposition}
\end{equation} 
where the only nonzero entries of $\mathcal{D}$ are $\mathcal{D}_{i,\ldots,i}$ for $1\leq i\leq r$. If $\mathcal{D}$ is omitted in this bracket notation, then by convention all the nonzero entries of $\mathcal{D}$ are 1. This tensor decomposition can be stored using $\red{\cpsto \le r+r \sum_{k=1}^d n_k}$ \red{degrees of freedom}, but the decomposition is NP-hard to compute for worst case examples~\cite{haastad1990tensor}. The CP decomposition in~\cref{eq:CPdecomposition} is similar to the orthogonal Tucker decomposition with two important differences: (1) The matrices $A^{(1)},\ldots,A^{(d)}$ in~\cref{eq:CPdecomposition} do not need to have orthogonal columns and (2) The core tensor $\mathcal{D}$ must be diagonal. This means that~\cref{eq:CPdecomposition} is equivalent to expressing a tensor as a sum of rank-1 tensors.   

Since we are aiming for upper bounds on the rank of a tensor \red{to bound compressibility of a tensor in CP format}, we can take any decomposition of the form in~\cref{eq:CPdecomposition} with a potentially large $r$, and see if its factor matrices $A^{(1)},\dots,A^{(d)}$ are themselves low rank. For example, we find that~\cite[Lem.~1]{kruskal1988rank}:\footnote{Lemma 1 of~\cite{kruskal1988rank} shows that the dimension of the vector space that spans the slices in the $\nu$th index is equal to the rank of $\mathcal{X}$. The inequality in~\cref{eq:Kruskal} follows from the extra assumption that the slices are themselves low rank tensors.}

\begin{equation}
\rank(\mathcal{X}) \leq \min_{1 \le j \le d} \frac{1}{r_j} \prod_{i=1}^d r_i,
\label{eq:Kruskal} 
\end{equation} 
where $r_i={\rm rank}(A^{(i)})$ for $1 \le i \le d$.  The bound in~\cref{eq:Kruskal} is useful because it allows one to derive upper bounds on the rank of a tensor via bounds on the rank of factor matrices.

\section{Tensors derived by sampling smooth functions}
\red{One often finds that tensors derived from sampling smooth functions are compressible, and we make this observation precise. Tensors derived from sampling multivariate functions have been considered in~\cite{khoromskij2010fast,ibragimov2009three} and analogous results for matrices are available in the literature~\cite{reade1983eigenvalues,townsend2014computing}.}

\subsection{Tensors constructed via sampling algebraically structured functions} \label{sec:algebraic}
In practice, one often encounter tensors that are sampled from multivariate functions. For example, one can take a continuous function of three variables, $f(x,y,z)$, and sample $f$ on a tensor grid to obtain a tensor: 
\[
\mathcal{X}_{ijk} = f(x_i,y_j,z_k), \qquad 1\leq i,j,k\leq n, 
\]
where $\{x_1,\ldots,x_n\}$, $\{y_1,\ldots,y_n\}$, and $\{z_1,\ldots,z_n\}$ are sets of points. 


\subsubsection{Polynomials and algebraic structure}
One common scenario where it is easy to spot \red{compressible tensors} is when the tensor is sampled from a polynomial. To be specific, if a tensor $\mathcal{X}$ is derived by sampling a multivariate polynomial $p(x_1,\dots,x_d)$ of degree at most $N_j-1$ in the variable $x_j$ from a tensor-product grid, \red{then one finds that $\mathcal{X}$ is highly compressible.}

\begin{lemma}
Let $p(x_1,\ldots,x_d)$ be a polynomial of degree at most $N_j-1$ in the variable $x_j$ for $1 \le j \le d$, and let $\mathcal{X} \in \C^{n_1 \times \dots \times n_d}$ be the tensor constructed by sampling $p$, i.e.,
\[ 
\mathcal{X}_{i_1,\dots,i_d} = p(x_{i_1}^{(1)},\dots,x_{i_d}^{(d)}), \qquad 1\leq i_j \leq n_j, 
\qquad 1 \le j \le d,
\]
where $x^{(1)},\dots,x^{(d)}$ are sets of $n_1,\dots,n_d$ nodes, respectively. Then,
\begin{itemize}[leftmargin=*,noitemsep]
\item \red{$\ttsto(\mathcal{X}) \le \sum_{k=1}^d t_{k-1}t_k n_k$}, {\rm where} $t_k=\min\{\prod_{j=1}^k N_j,\prod_{j=k+1}^d N_j\}$,
\item \red{$\mlsto(\mathcal{X}) \le \sum_{k=1}^{d} n_k N_k + \prod_{k=1}^{d} N_k$}, and
\item \red{$\cpsto(\mathcal{X}) \le r+r \sum_{k=1}^d n_k$, where} $r =\min_{1 \le k \le d} \frac{1}{N_k} \prod_{j=1}^d N_j$.
\end{itemize}
Here, the tensor-train decomposition is constructed in the order of $x_1,\dots,x_d$.
\label{lem:polyRanks} 
\end{lemma}
\begin{proof}
According to the degree assumptions on $p$, we can write $p$ as 
\[
p(x_1,\dots,x_d)=\sum_{q_1=0}^{N_1-1} \cdots \sum_{q_k=0}^{N_k-1} a_{q_1,\dots,q_k}(x_{k+1},\dots,x_{d}) x_1^{q_1} \cdots x_k^{q_k}, \qquad 1 \le k \le d,
\]
where $a_{q_1,\dots,q_k}(x_{k+1},\dots,x_{d})$ is a polynomial in the variables $x_{k+1},\dots,x_{d}$ and for $k+1\le j \le d$, $x_j$ has degree at most $N_j-1$. After sampling, this means that ${\rm rank}(X_k) \le \min\{\prod_{j=1}^k N_j,\prod_{j=k+1}^d N_j\}$ and the bound on \red{$\ttsto(\mathcal{X})$} follows.

Another way to write $p$ is
\[ 
p(x_1,\dots,x_d)=\sum_{j=0}^{N_k-1} b_j(x_1,\dots,x_{k-1},x_{k+1},\dots,x_d)x_k^j,\qquad 1 \le k \le d,
\] 
where $b_j$ is a polynomial in $x_1,\dots,x_{k-1},x_{k+1},\dots,x_d$ of degree at most $N_j-1$ in $x_j$. After sampling, this shows that $\rank(X_{(k)}) \le N_k$ and the bound on \red{$\mlsto(\mathcal{X})$} follows.

Finally, separating out $x_d$, we can also write $p$ as
\begin{equation} 
p(x_1,\dots,x_d)=\sum_{q_1=0}^{N_1-1} \cdots \sum_{q_{d-1}=0}^{N_{d-1}-1} c_{q_1,\dots,q_{d-1}}(x_{d})x_1^{q_1} \cdots x_{d-1}^{q_{d-1}}, 
\label{eq:CPpoly} 
\end{equation} 
where each term in~\cref{eq:CPpoly} is a rank 1 tensor after sampling. We can do this to each variable and thus $\rank(\mathcal{X}) \le \min_{1 \le k \le d} \frac{1}{N_k} \prod_{j=1}^d N_j$. \red{The bound on $\cpsto(\mathcal{X})$ follows.}
\end{proof}

A special case of~\cref{lem:polyRanks} is when the polynomial $p$ has maximal degree of at most $N-1$ so that $N_1=\cdots=N_d=N$.\footnote{We say that a polynomial $p_N(x_1,\ldots,x_d)$ has maximal degree $\leq N$ if $p_N$ is a polynomial of degree at most $N$ in all the variables $x_i$.} \red{We find that}
\begin{itemize}[leftmargin=*,noitemsep]
\item \red{$\ttsto(\mathcal{X}) \le \sum_{k=1}^d N^{2t-1} n_k$, where $t = \min\{k,d-k\}$},
\item \red{$\mlsto(\mathcal{X}) \le N\sum_{k=1}^d n_k + N^d$}, and
\item \red{$\cpsto(\mathcal{X}) \le N^{d-1}\sum_{k=1}^d n_k + N^{d-1}$.}
\end{itemize}
\red{The important observation is that tensors constructed by sampling polynomials on a grid are highly compressible.}

\subsubsection{Other special cases of algebraic structure}
Similar to multivariate polynomials, it is easy to spot---after some experience---the mathematical tensor ranks of tensors constructed by sampling functions that have explicit algebraic structure since each variable in the function can be thought as a fiber of the tensor. The easiest ones to spot are those tensors derived from functions that are the sums of products of single-variable functions, such as
\[ 
f(x,y,z)=1+\tan(x)y+y^2z^3, \quad g(x,y,z,w)=\cos(x)\sin(y)+e^{10z}e^{100w }. 
\]
If $\mathcal{F}$ and $\mathcal{G}$ are tensors constructed by sampling $f$ and $g$ on a \red{$n \times n \times n$ and $n \times n \times n \times n$} tensor-product grid, respectively, then the \red{storage costs in different formats} are bounded by 
\[ 
\begin{aligned} 
&\ttsto(\mathcal{F}) \le 8n, \quad &\mlsto(\mathcal{F}) &\le 7n+12, \quad &\cpsto(\mathcal{F}) \le 9n+3, \\
&\ttsto(\mathcal{G}) \le12n, \quad &\mlsto(\mathcal{G}) &\le 8n+16, \quad &\cpsto(\mathcal{G}) \le 8n+2,
\end{aligned} 
\]
where the tensor-train decompositions are performed in the order $x,y,z,w$. Other examples are functions that can be expressed with exponentials and powers, \red{and similar examples have also been considered~\cite{oseledets2013,khoromskij2011dlog}}. 

Some special functions require reorganizations to reveal their algebraic structures. If the function is expressed with trigonometric functions, then the sampled tensor can often be low rank due to trigonometric identities. For example, \red{consider the function $f(x,y,z)=\cos(x+y+z)$ that is a special case of the examples in~\cite{oseledets2009breaking,beylkin2002numerical}}. Since it can be written as 
\[ 
f(x,y,z) = (\cos (x) \!\cos (y) - \sin (x)\! \sin (y) )\!  \cos (z) - (\sin (x)\!\cos (y) + \cos (x)\!\sin (y)) \!\sin (z), 
\]
any tensor $\mathcal{F}$ constructed by sampling $f$ on a \red{$n \times n \times n$} tensor-product grid satisfies
\[ 
\ttsto(\mathcal{F}) \le 8n, \quad \mlsto(\mathcal{F}) \le 6n+8, \quad \cpsto(\mathcal{F}) \le 12n+4. 
\]
These examples can often be combined to build more complicated functions that result in \red{compressible} tensors. This is an $ad \ hoc$ process and requires human ingenuity to express the sampled function in a revealing form. Again, tensors constructed by sampling such algebraically structured functions on a sufficiently large tensor-product grid can be represented using a small number of degrees of freedom.

\subsection{Tensors derived by sampling smooth functions}
\label{sec:smoothness}  
Although most functions do not have the algebraic structure specified in~\cref{sec:algebraic}, tensors that are constructed by sampling smooth functions are often well approximated by \red{compressible} tensors. In light of~\cref{lem:polyRanks}, our idea to understand the \red{compressibility} of a tensor derived from sampling a function is first to approximate that function by a multivariate polynomial, which is already a routine procedure for computing with low rank approximations to multivariate functions~\cite{hashemi2017chebfun}. 

Without loss of generality, suppose that $\mathcal{X}$ is formed by sampling a smooth function $f$ on a tensor-product grid in $[-1,1]^d$, i.e., 
\begin{equation} 
\mathcal{X}_{i_1,\ldots,i_d} = f(x_{i_1}^{(1)},\dots,x_{i_d}^{(d)}),\qquad \red{1\leq i_k \leq n_k, \qquad 1 \le k \le d},
\label{eq:sampling} 
\end{equation} 
where $x^{(1)},\dots,x^{(d)}$ are sets of $n_1,\dots,n_d$ nodes in $[-1,1]$. Our idea is to find a multivariate polynomial $p$ of degree $\leq N_j-1$ in the variable $x_j$ that approximates $f$ in $[-1,1]^d$ and then set $\mathcal{Y}_{i_1,\ldots,i_d} = p(x_{i_1}^{(1)},\dots,x_{i_d}^{(d)})$. By~\cref{lem:polyRanks}, $\mathcal{Y}$ can be represented with a small number of degrees of freedom and $\mathcal{Y}$ is an approximation to $\mathcal{X}$. In particular, we have
\begin{equation} 
\|\mathcal{X}-\mathcal{Y}\|_F \leq\! \left(\prod_{i=1}^d n_i\right)^{\!\!\frac{1}{2}}\! \|\mathcal{X}-\mathcal{Y}\|_{{\rm max}} \leq \!\left(\prod_{i=1}^d n_i\right)^{\!\!\frac{1}{2}}\! \|f-p_N\|_\infty,
\label{eq:functionBound} 
\end{equation} 
where $\|\cdot\|_\infty$ denotes the supremum norm on $[-1,1]^d$ and $\|\cdot\|_{\rm max}$ is the absolute maximum entry norm. Therefore, if $p$ is a good approximation to $f$, then $\mathcal{Y}$ is a good approximation to $\mathcal{X}$ too. \red{Although the error bound is good for small $d$, this approximation still suffers from the curse of dimensionality for large $d$.} 

One can now propose any linear or nonlinear approximation scheme to find a polynomial approximation $p$ of $f$ on $[-1,1]^d$. Clearly, excellent bounds on $\|\mathcal{X}-\mathcal{Y}\|_F$ are obtained by finding a $p$ so that 
\[
\| f - p\|_\infty \approx \inf_{q\in \mathcal{P}_{N_1,\dots,N_d}} \| f - q \|_\infty,
\]
where $\mathcal{P}_{N_1,\dots,N_d}$ is the space of $d$-variate polynomials of maximal degree $\le N_i-1$ in $x_i$ for $1 \le i \le d$. This best multivariable polynomial problem is often, but not always, tricky to solve directly. \red{In those cases, near-optimal polynomial approximations are used instead. One common choice is to use} $p$ as the multivariate Chebyshev projection of $f$. That is, 
\[
\begin{aligned} 
p_{N_1,\dots,N_d}^{\rm cheb}(x_1,\ldots,x_d) &= \sum_{i_1=0}^{N_1-1}\!\!{}^{'}\cdots \sum_{i_d=0}^{N_d-1}\!\!{}^{'} c_{i_1,\ldots,i_d} T_{i_1}(x_1)\cdots T_{i_d}(x_d), \\
c_{i_1,\ldots,i_d} &= \left(\frac{2}{\pi}\right)^d\!\!\int_{-1}^1 \cdots \int_{-1}^1 \frac{f(x_1,\ldots,x_d)T_{i_1}(x_1)\cdots T_{i_d}(x_d)}{\sqrt{1-x_1^2}\cdots \sqrt{1-x_d^2}}dx_1\cdots dx_d,
\end{aligned} 
\]
where the primes indicate that the first term in the sum is halved and $T_k(x)$ is the Chebyshev polynomial of degree $k$. Importantly, $p_{N_1,\dots,N_d}^{\rm cheb}$ is a near-best polynomial approximation to $f$~\cite{trefethen2013approximation}, and the error $\smash{\| f - p_{N_1,\dots,N_d}^{\rm cheb}\|_\infty}$ can be bounded.  Thus, this choice of $p$ leads to bounds on the compressibility of $\mathcal{X}$.

\subsection{Worked examples} 
\label{sec:result}
Here, we give two examples that illustrate how to understand the compressibility of tensors constructed by sampling smooth functions. We consider two functions: (1) A Fourier-like function, where we use best polynomial approximation, and (2) A sum of Gaussian bumps, where we use Chebyshev approximation.

\subsubsection{Fourier-like function} 
Consider a tensor $\mathcal{X} \in \C^{n \times n \times n}$ constructed by sampling the following Fourier-like function on a tensor-product grid~\cite{townsend2014computing}:
\[ 
f(x,y,z)=e^{iM\pi xyz}, \qquad x,y,z \in [-1,1], 
\]
where $M \geq 1$ is a real parameter. While \red{representing $\mathcal{X}$ exactly requires $n^3$ degrees of freedom, it can be approximated by tensors that require fewer degrees of freedom (in the tensor-train and Tucker formats)}. To see this, let $p^{{\rm best}}_{k-1}$ and $q^{{\rm best}}_{k-1}$ be the best minimax polynomial approximations of degree $\leq k-1$ to $\cos(M\pi t)$ and $\sin(M\pi t)$ on $[-1, 1]$, respectively, and define $h_{k-1} = p^{{\rm best}}_{k-1}+iq^{{\rm best}}_{k-1}$. Note that $h_{k-1}(xyz)$ has maximal degree at most $k-1$ so that
\[ 
\begin{aligned} 
\inf_{w_{k-1}\in\mathcal{P}_{k-1}} \sup_{x,y,z\in[-1,1]}\left|e^{iM\pi xyz} - w_{k-1}(x,y,z)\right| &\leq \sup_{x,y,z\in[-1,1]}\left|e^{iM\pi xyz} - h_{k-1}(xyz)\right| \\
&= \sup_{t\in[-1,1]}\left|e^{iM\pi t} - h_{k-1}(t)\right|, 
\end{aligned} 
\]
where $\mathcal{P}_{k-1}$ is the space of trivariate polynomials of maximal degree $\leq k-1$ and the equality follows since $t = xyz \in [-1, 1]$ if $x,y,z\in[-1,1]$. Furthermore, we have $e^{iM\pi t} = \cos(M\pi t) + i \sin(M\pi t)$ and so
\[
\sup_{t\in[-1,1]}\!\left| e^{iM\pi t} - h_{k-1}(t)\right| \leq \!\sup_{t\in[-1,1]}\!\left|\cos(M\pi t) - p^{{\rm best}}_{k-1}(t)\right|+\!\sup_{t\in[-1,1]}\!\left|\sin(M\pi t) - q^{{\rm best}}_{k-1}(t)\right|. \]
By the equioscillation theorem~\cite[Thm.~7.4]{powell1981approximation}, $p^{{\rm best}}_{k-1} = 0$ for $k-1 \le 2\floor{M} - 1$ since $\cos(M\pi t)$ equioscillates $2\floor{M} + 1$ times in $[-1, 1]$. Similarly, $\sin(M\pi t)$ equioscillates $2\floor{M}$ times in $[-1, 1]$ and hence, $q^{{\rm best}}_{k-1}=0$ for $k-1 \le 2\floor{M}-2$. However, for $k > 2\floor{M}$, $\sup_{t\in[-1,1]}\!\left| e^{iM\pi t} - h_{k-1}(t)\right|$ decays super-geometrically to zero as $k \rightarrow \infty$. \red{This also indicates that the error between the tensors sampled from $e^{iM\pi xyz}$ and $h_{k-1}(x,y,z)$ rapidly goes to 0 as $k \rightarrow \infty$.} Hence, the numerical maximal degree, $N_\epsilon$, of $e^{iM\pi xyz}$ satisfies $N_\epsilon/2M \rightarrow c$ for some constant $c \geq 1$ as $M \rightarrow \infty$.~\Cref{lem:polyRanks} shows \red{that an approximant to $\mathcal{X}$ only requires $\mathcal{O}(M)$ degrees of freedom.} \red{In particular, if $s_1$ is the second element of the tensor-train rank of an approximant tensor to the one sampled by $e^{iM\pi xyz}$}, then $s_1/(2M)\rightarrow 1$ as $M\rightarrow \infty$.  

Figure~\ref{fig:smooth_ex} (left) plots the ratio of the second element of the tensor-train rank, $s_1$, of a tensor sampled from the Fourier-like function and $2M$. We observe that $s_1/(2M) \rightarrow 1$ as $M\rightarrow \infty$.  

\begin{figure} 
\begin{minipage}{0.49\textwidth}
\begin{overpic}[width=\textwidth]{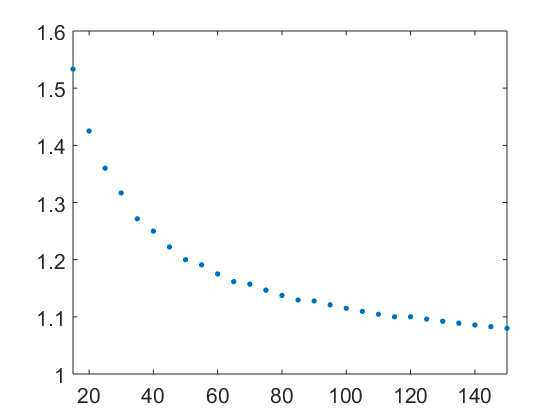}
\put(50,-2) {$M$}
\put(0,35) {\rotatebox{90}{ratio}}
\put(42,72) {$s_1/(2M)$}
\end{overpic} 
\end{minipage}
\begin{minipage}{0.49\textwidth}
\begin{overpic}[width=\textwidth]{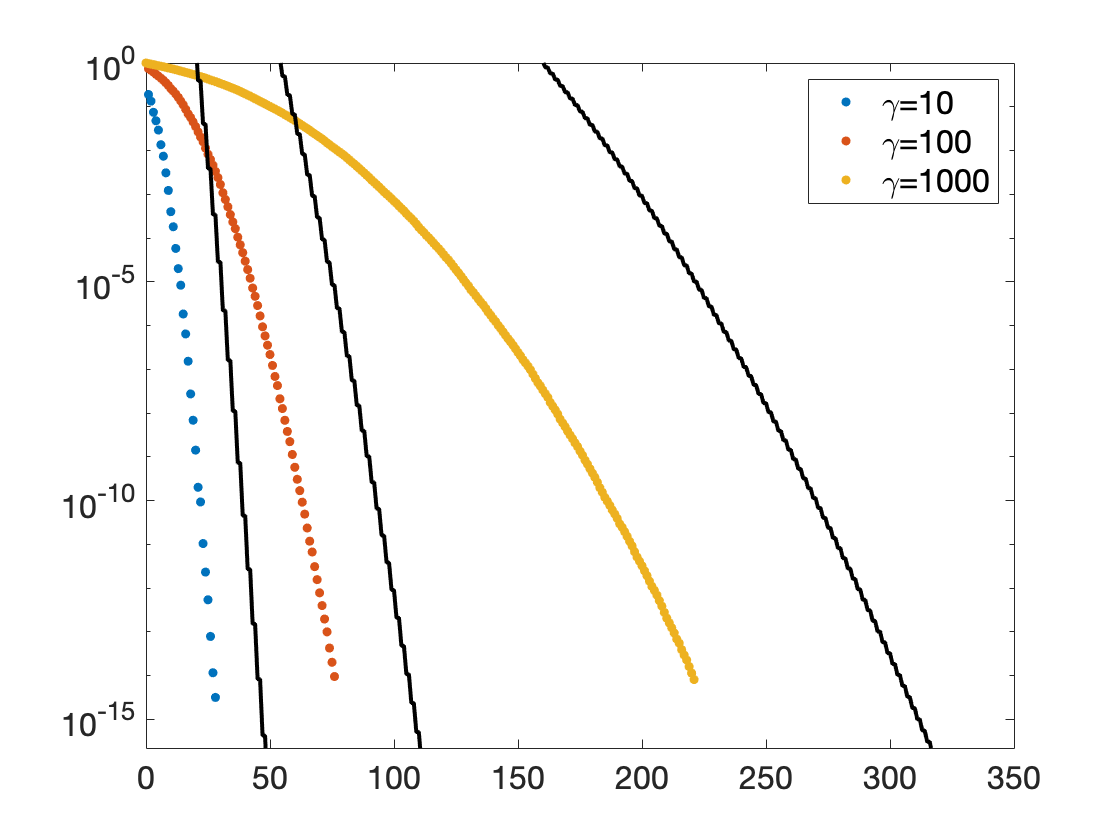}
\put(43,-2) {TT-rank}
\put(0,28) {\rotatebox{90}{Accuracy}}
\put(66,30) {\rotatebox{-66}{$\gamma=1000$}}
\put(35,30) {\rotatebox{-78}{$\gamma=100$}}
\put(23,26) {\rotatebox{-85}{$\gamma=10$}}
\end{overpic} 
\end{minipage}
\caption{Left: The ratio of the second element of the tensor-train rank, $s_1$, of the tensors of size $n\times n\times n$ with $n = 600$ constructed by sampling the Fourier-like function $e^{iMxyz}$ with $15 \leq M \leq 150$. The accuracy used to calculate the tensor-train ranks is $10^{-10}$. Right: The \red{second element, $s_1$, of the} tensor-train rank (blue, red, and yellow dots) \red{calculated with the TT-SVD} and the theoretical bounds (black lines) of $\mathcal{X} \in \C^{400 \times 400 \times 400}$ constructed by sampling $\smash{\sum_{j=1}^{300} e^{-\gamma((x-x_j)^2+(y-y^j)^2+(z-z_j)^2)}}$ on an equispaced tensor-product grid for $\gamma=10,100,1000$, where $(x_j,y_j,z_j)$ are arbitrary centers in $[-1,1]^3$.}
\label{fig:smooth_ex} 
\end{figure}
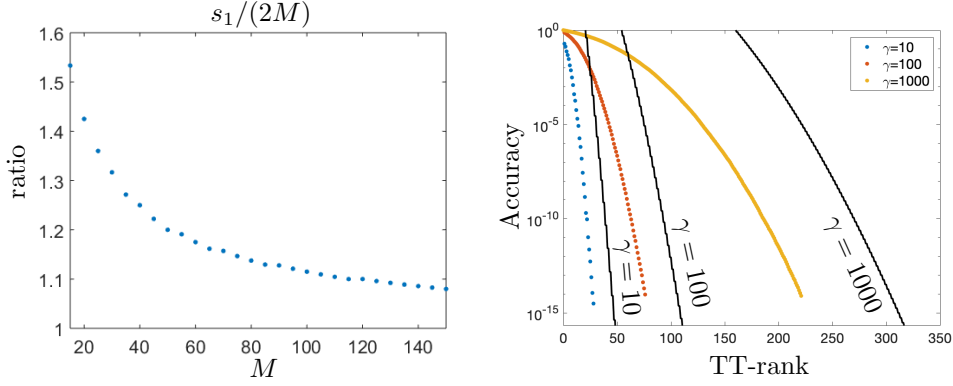

\subsubsection{A sum of Gaussian bumps}\label{ex:gauss_bumps}
Consider a tensor $\mathcal{X} \in \C^{n \times n \times n}$ constructed by sampling a sum of $M$ Gaussian bumps, centered at arbitrary locations $(x_1,y_1,z_1),\dots,(x_M,y_M,z_M)$ in $[-1,1]^3$, i.e.,
\begin{equation} 
f(x,y,z)=\sum_{j=1}^M e^{-\gamma((x-x_j)^2+(y-y_j)^2+(z-z_j)^2)}, \qquad \gamma>0. 
\label{eq:Bumps} 
\end{equation} 
Each Gaussian bump is a separable function of three variables so, mathematically, the tensor ranks of $\mathcal{X}$ depend linearly on $M$. However, since the sum is a smooth function, the ranks are related to the polynomial degree required to approximate $f(x,y,z)$ in $[-1,1]^3$ to an accuracy of $0<\epsilon<1$. Hence, the tensor ranks of $X$ depend on $\gamma$ and have very mild growth 
in $M$ in the storage costs. 

Due to the symmetry in $x$, $y$, and $z$ as well as separability of each term in~\cref{eq:Bumps}, we find that the Chebyshev approximation to $f(x,y,z)$ can be bounded by
\[
\sup_{x,y,z\in[-1,1]} \left|f(x,y,z) - \sum_{j=1}^M p_{\ell}^j(x)q_{\ell}^j(y)r_{\ell}^j(z)\right|\leq 3M \!\sup_{x\in[-1,1]} \left| e^{-\gamma x^2} - h_\ell(x)\right|,
\]
where $p_{\ell}^j$, $q_\ell^j$, $r_\ell^j$, and $h_\ell$ are Chebyshev approximations of degree $\leq \ell$ to $\smash{e^{-\gamma(x-x_j)^2}}$, $\smash{e^{-\gamma(y-y_j)^2}}$, $\smash{e^{-\gamma(z-z_j)^2}}$, and $\smash{e^{-\gamma x^2}}$, respectively. An explicit Chebyshev expansion for $\smash{e^{-\gamma x^2}}$ is known and given by~\cite[p.~32]{luke1969special}
\[
e^{-\gamma x^2} = \sum_{j=0}^\infty{}^{'} (-1)^j e^{-\gamma/2} I_j(\gamma/2) T_{2j}(x),  
\]
where the prime on the summation indicates that the first term is halved, and $I_j(z)$ is the modified Bessel function of the first kind with parameter $j$~\cite[(10.25.2)]{olver2010nist}. This means that one can show that~\cite[Lem.~5]{gardner2018gpytorch}:\footnote{Unfortunately, there is a typo in~\cite[Lem.~5]{gardner2018gpytorch} and $I_{\ell+1}(\gamma/4)$ should be replaced by $I_{\floor{\ell/2}+1}(\gamma/4)$.}

\[ h_\ell(x) = \!\sum_{j=0}^{\ell}\!{}^{'} (-1)^j e^{-\gamma/2} I_j(\gamma/2) T_{2j}(x), \ \sup_{x\in[-1,1]}\!\left|e^{-\gamma x^2}-h_\ell(x)\right| \leq 2e^{-\gamma /4} I_{\floor{\ell/2}+1}(\gamma/4). 
\]

By~\cref{lem:polyRanks} and~\cref{eq:functionBound}, \red{we can understand the compressibility of $\mathcal{X}$. In particular, we can find an approximant tensor whose tensor-train ranks are} bounded by the smallest integer $\ell$ such that $6Mn^{3/2}e^{-\gamma /4} I_{\floor{\ell/2}+1}(\gamma/4) \le \epsilon$.  \red{We find it straightforward to visualize compressibility via elements of the tensor ranks and their bounds, due to the way storage costs are calculated.} Figure~\ref{fig:smooth_ex} (right) shows the second element of the tensor-train rank, $s_1$ of the approximant tensor, along with the bound that we derived. The bounds are relatively tight when $\epsilon$ is small.

\section{Tensors with displacement structure} \label{sec:displacement}
We say that $\mathcal{X}\in\mathbb{C}^{n_1\times\cdots\times n_d}$ has an $(A_1,\ldots,A_d)$-displacement structure of $\mathcal{G}\in\mathbb{C}^{n_1\times\cdots\times n_d}$ if $\mathcal{X}$ satisfies the multidimensional Sylvester equation
\begin{equation} \label{eq:tensor_disp}
\mathcal{X} \times_1 A^{(1)} + \cdots + \mathcal{X} \times_d A^{(d)} = \mathcal{G}, \qquad A^{(k)} \in \C^{n_k \times n_k},
\end{equation}
where `$\times_k$' is the $k$-mode matrix product of a tensor. In this section, we show that when $A_1,\ldots,A_d$ are normal matrices with ``separated" spectra and $\mathcal{G}$ is a low rank tensor, then $\mathcal{X}$ \red{is compressible}. Several classes of structured tensors (e.g., the Hilbert tensor) and the solution tensors of certain discretized partial differential equations (e.g., the discretized solution to Poisson's equation) have a displacement structure, which leads to an understanding of their compressibility.

\subsection{Zolotarev numbers} \label{sec:zolo}
The bounds that we derive on compressibility of tensors involve so-called Zolotarev numbers~\cite{akhiezer1990elements,gonvcar1969zolotarev,zolotarev1877application}. A Zolotarev number is a positive number between $0$ and $1$ defined via an infimum problem involving rational functions~\cite{zolotarev1877application}. Namely,
\begin{equation} \label{eq:zol_def}
Z_k(E,F) := \inf_{r \in \mathcal{R}_{k,k}} \frac{\sup_{z \in E} |r(z)|}{\inf_{z \in F} |r(z)|},\qquad k\geq 0,
\end{equation}
where $E$ and $F$ are disjoint complex sets and $\mathcal{R}_{k,k}$ is the set of irreducible rational functions of the form $p(x)/q(x)$ with polynomials $p$ and $q$ of degree at most $k$. If $E$ and $F$ are well-separated, then one finds that $Z_k(E,F)$ decays rapidly with $k$. This is because one can construct a low degree rational function that is small on $E$ and large on $F$. If $E$ and $F$ are close to each other, then typically $Z_k(E,F)$ decreases much more slowly with $k$. 

Zolotarev numbers can be used to bound the singular values of matrices with displacement structure~\cite[Thm.~2.1]{beckermann2017singular}. In particular, if $X \in \C^{m \times n}$ with $m \ge n$ satisfies the displacement structure
\begin{equation} 
AX-XB=MN^*, \qquad M \in \C^{m \times \nu}, \quad N \in \C^{n \times \nu}, 
\label{eq:displacementStructure}
\end{equation} 
where $A\in\mathbb{C}^{m\times m}$ and $B\in\mathbb{C}^{n\times n}$ are normal matrices with spectra $\Lambda(A) \subseteq E$ and $\Lambda(B) \subseteq F$, then the singular values of $X$ satisfy~\cite[Thm.~2.1]{beckermann2017singular}
\begin{equation} 
\sigma_{j+\nu k}(X) \le Z_k(E,F)\sigma_j(X), \quad 1 \le j+\nu k \le n.
\label{eq:SVzolotarev}
\end{equation} 
Roughly speaking, if $\Lambda(A)$ and $\Lambda(B)$ are well-separated and $\nu$ is small, then the singular values $\sigma_j(X)$ decrease rapidly to $0$. 

When working with tensors, we translate the inequalities in~\cref{eq:SVzolotarev} into Frobenius norm error bounds \red{so that matrix results can then be utilized}.

\begin{lemma} 
If $X\in\mathbb{C}^{m\times n}$ is a matrix satisfying~\cref{eq:displacementStructure} and $X_{\nu k}$ is the best rank $\nu k$ approximation to $X$, then 
\[
\|X - X_{\nu k} \|_F \leq Z_k(E,F) \|X\|_F,
\]
where $\|\cdot \|_F$ denotes the matrix Frobenius norm. 
\label{lem:FrobNorm} 
\end{lemma} 
\begin{proof} 
To simplify notation let $Z_k=Z_k(E,F)$, $r=\nu k$, $\sigma_j = \sigma_j(X)$ for $1\leq j\leq n$, and $\sigma_j = 0$ for $j>n$. \red{If $k = 0$, then $r = 0$ and $Z_k = 1$, so $X_r = 0$ and the statement follows automatically. Now consider $k > 0$}, note that for any $s\geq 1$ we have
\[
\sum_{j=sr+1}^{(s+1)r} \sigma_j^2 = \sum_{j=1}^{r} \sigma_{j+sr}^2 \leq Z_k^2 \sum_{j=1}^{r} \sigma_{j+(s-1)r}^2 \leq \cdots \leq Z_k^{2s} \sum_{j=1}^{r} \sigma_{j}^2,
\]
where the inequalities come from the repeated application of the bound in~\cref{eq:SVzolotarev}. 
Therefore, we can bound $\|X-X_r\|_F^2$ by partitioning the singular values into groups of $r$. That is, 
\[
\|X-X_r\|_F^2 = \sum_{j=r+1}^n \sigma_i^2 \leq \sum_{s=1}^{\infty} \sum_{j=sr+1}^{(s+1)r} \sigma_i^2 \leq \sum_{s=1}^{\infty} Z_k^{2s} \sum_{j=1}^{r} \sigma_{j}^2 = \frac{Z_k^2}{1-Z_k^2}\sum_{j=1}^{r}\sigma_{j}^2,
\]
where the last equality is obtained by summing up the geometric series. Since $ \|X\|_F^2 = \sum_{j=1}^n \sigma_j^2$, we find that 
\[
\left(1 + \frac{Z_k^2}{1-Z_k^2} \right) \|X-X_r\|_F^2\leq \frac{Z_k^2}{1-Z_k^2} \|X\|_F^2.
\]
The result follows by rearranging. 
\end{proof} 

For $0<\epsilon<1$, the numerical rank of $X$ measured in the Frobenius norm is the smallest integer, $r_\epsilon$, such that 
\[
\inf_{{\rm rank} (\tilde{X}) \leq r_\epsilon } \|X - \tilde{X}  \|_F \leq \epsilon \left\|X\right\|_F.
\]
We denote this integer by $\rank_\epsilon(X)$. From~\cref{lem:FrobNorm}, we find that for matrices that satisfy~\cref{eq:displacementStructure}, we have
\begin{equation} \label{eq:zolo_fro}
\rank_\epsilon(X) \le \nu k,
\end{equation}
where $k$ is the smallest integer so that $Z_k(E,F) \le \epsilon$. Therefore, Zolotarev numbers are very useful when trying to bound the numerical rank of matrices with displacement structure. For example, for an $n \times n$ Pick matrix $P_n$ constructed with real numbers from an inverval $[a,b]$ with $0 < a < b < \infty$, one can find that $\rank_\epsilon(P_n) \le \red{2}\Bigl\lceil \log(4b/a)\log(4/\epsilon)/\pi^2 \Bigr\rceil$~\cite{beckermann2017singular}.

\subsection{The compressibility of tensors with displacement structure in the tensor-train format} \label{sec:tt}
Zolotarev numbers can also be used to understand the compressibility of tensors satisfying~\eqref{eq:tensor_disp}. From the bounds in~\cref{eq:TT_trivial}, one finds that the numerical ranks of each unfolding provides an upper bound on \red{all entries of the tensor-train ranks of approximant tensors}. More precisely, if $\mathcal{X} \in \C^{n_1 \times \dots \times n_d}$ is a tensor and $0 < \epsilon < 1$, then there exists a tensor $\tilde{\mathcal{X}}$ such that~\cite[Thm.~2.2]{oseledets2011tensor}
\begin{equation} \label{eq:num_ttrank}
||\mathcal{X}-\tilde{\mathcal{X}}||_F \le \epsilon ||\mathcal{X}||_F, \quad \ttrank(\tilde{\mathcal{X}})=(1,\rank_\delta(X_1),\dots,\rank_\delta(X_{d-1}),1),
\end{equation}
where $\delta=\epsilon/\sqrt{d\red{-1}}$ and $X_k$ is the $k$th unfolding of $\mathcal{X}$. \red{In order to easily relate tensor-train ranks with multilinear ranks in the next subsection, we choose to use $\delta=\epsilon/\sqrt{d}$.}

If $\mathcal{X}$ satisfies~\cref{eq:tensor_disp}, then by rearranging~\cref{eq:tensor_disp} one can show that each unfolding matrix, $X_j$, has a displacement structure. This is precisely $B_jX_j-X_jC_j^T=G_j$, where $G_j$ is the $j$th unfolding of $\mathcal{G}$ and 
\[
\begin{aligned} 
B_j &= I \otimes \cdots \otimes I \otimes A^{(1)}+\cdots+A^{(j)} \otimes I \otimes \cdots \otimes I, \\
C_j &=-(I \otimes \cdots \otimes I \otimes A^{(j+1)}+\cdots+A^{(d)} \otimes I \otimes \cdots \otimes I).
\end{aligned} 
\]
From properties of the Kronecker product~\cite[Thm 2.5]{schacke2004kronecker}, we know that $B_j$ and $C_j$ are normal matrices with $\Lambda(B_j)=\Lambda(A^{(1)})+ \dots +\Lambda(A^{(j)})\subseteq E_j$ and $\Lambda(C_j)=-(\Lambda(A^{(j+1)})+ \dots +\Lambda(A^{(d)}))\subseteq F_j$.\footnote{By $\Lambda(A)+\Lambda(B)$ we mean the Minkowski sum, formed by adding each element in $\Lambda(A)$ to each element in $\Lambda(B)$, i.e., \[ \Lambda(A)+\Lambda(B) = \{ a+b \ | \ a \in \Lambda(A), \ b \in \Lambda(B)\}. \]} From~\cref{lem:FrobNorm} we see that for any integer $k_j$ such that $Z_{k_j}(E_j,F_j)\leq \delta$, then 
\[
{\rm rank}_\delta(X_j) \leq k_j\nu_j, \qquad \nu_j = {\rm rank}(G_j), \qquad 1\leq j\leq d-1.
\]
Therefore, a necessary condition to bound the numerical tensor-train ranks of $\mathcal{X}$ using this approach is that the spectra of $A^{(1)},\ldots,A^{(d)}$ are {\em Minkowski sum separated}. 

\begin{definition} \label{def:minkowski_separated}
We say that normal matrices $A^{(1)},\ldots,A^{(d)}$ are Minkowski sum separated if there are disjoint sets $E_j$ and $F_j$ so that 
\[ 
\Lambda(A^{(1)})+\dots+\Lambda(A^{(j)}) \subseteq E_j, \quad -(\Lambda(A^{(j+1)})+\dots+\Lambda(A^{(d)})) \subseteq F_j, \quad 1 \le j \le d-1, 
\]
where the set additions are Minkowski sums and $\Lambda(A^{(j)})$ denotes the spectrum of $A^{(j)}$.
\end{definition}

\Cref{fig:min_sum} illustrates three Minkowski sum separated matrices $A^{(1)},A^{(2)}$, and $A^{(3)}$ along with possible choices for the sets $E_j$ and $F_j$ for $j = 1,2$. We summarize our findings as a theorem. 
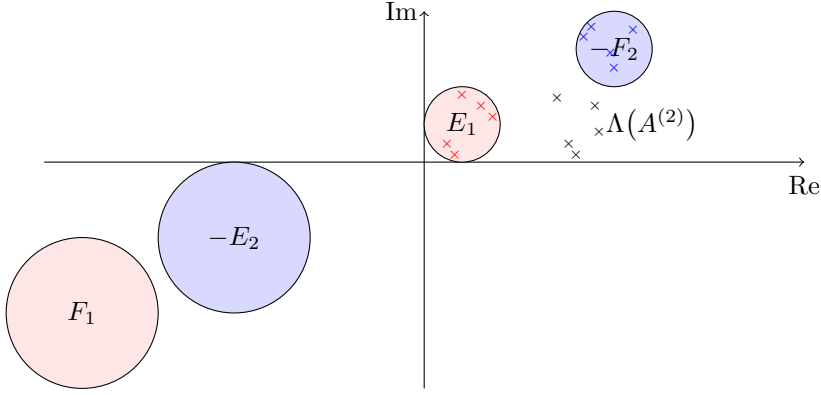
\begin{figure}
\label{fig:min_sum}
\centering
\begin{tikzpicture}
\draw[->] (-5,0)--(5,0);
\filldraw[black] (5,-0.3) node {Re};
\draw[->] (0,-3)--(0,2);
\filldraw[black] (-0.3,2) node {Im};
\filldraw[color=black,fill=pink!40] (0.5,0.5) circle (0.5);
\filldraw[black] (0.5,0.5) node {$E_1$};
\filldraw[red] (0.75,0.75) node {\tiny{$\times$}};
\filldraw[red] (0.9,0.6) node {\tiny{$\times$}};
\filldraw[red] (0.3,0.25) node {\tiny{$\times$}};
\filldraw[red] (0.4,0.1) node{\tiny{$\times$}};
\filldraw[red] (0.5,0.9) node{\tiny{$\times$}};
\filldraw[black] (3,0.5) node {$\Lambda\!\left( A^{(2)} \right)$};
\filldraw[black] (2.25,0.75) node {\tiny{$\times$}};
\filldraw[black] (1.75,0.85) node {\tiny{$\times$}};
\filldraw[black] (2.3,0.4) node {\tiny{$\times$}};
\filldraw[black] (2,0.1) node {\tiny{$\times$}};
\filldraw[black] (1.9,0.25) node {\tiny{$\times$}};
\filldraw[color=black,fill=blue!15] (2.5,1.5) circle (0.5);
\filldraw[black] (2.5,1.5) node {$-F_2$};
\filldraw[blue] (2.75,1.75) node {\tiny{$\times$}};
\filldraw[blue] (2.1,1.66) node {\tiny{$\times$}};
\filldraw[blue] (2.5,1.25) node {\tiny{$\times$}};
\filldraw[blue] (2.2,1.8) node {\tiny{$\times$}};
\filldraw[blue] (2.45,1.45) node {\tiny{$\times$}};
\filldraw[color=black,fill=pink!40] (-4.5,-2) circle (1);
\filldraw[black] (-4.5,-2) node {$F_1$};
\filldraw[color=black,fill=blue!15] (-2.5,-1) circle (1);
\filldraw[black] (-2.5,-1) node {$-E_2$};
\end{tikzpicture}
\caption{Minkowski sum separated matrices $A^{(1)},A^{(2)}$, and $A^{(3)}$ where the colored crosses denote the spectrum of $A^{(1)}$, $A^{(2)}$, and $A^{(3)}$, respectively. Here, $\Lambda(A^{(1)}) \subseteq E_1$, $\Lambda(A^{(3)}) \subseteq -F_2$, $\Lambda(A^{(1)})+\Lambda(A^{(2)}) \subseteq E_2$, and $\Lambda(A^{(2)})+\Lambda(A^{(3)}) \subseteq -F_1$. By definition, we must have that $E_1$ is disjoint from $F_1$ (red regions), and that $E_2$ is disjoint from $F_2$ (blue regions).}
\end{figure}

\begin{theorem} \label{thm:tt_disp}
Suppose $\mathcal{X} \in \C^{n_1 \times \dots \times n_d}$ satisfies~\cref{eq:tensor_disp}, where $A^{(1)},\dots,A^{(d)}$ are Minkowski sum separated with disjoint sets $E_j$ and $F_j$ for $1 \le j \le d-1$. Then, for a fixed $0 < \epsilon < 1$, we have
\[
\red{\ttsto_\epsilon(\mathcal{X}) \le \sum_{j = 1}^d (k_{d-1}\nu_{d-1})(k_d\nu_d)n_d}, \quad \nu_j=\rank(G_j),\quad 1 \le j \le d-1,
\]
where $G_j$ is the $j$th unfolding of $\mathcal{G}$ and $k_j$ is an integer so that $Z_{k_j}(E_j,F_j) \le \epsilon/\sqrt{d}$.
\end{theorem}

For special choices of $E_j$ and $F_j$, explicit bounds on $Z_{k_j}(E_j,F_j)$ are known and therefore the bounds in~\cref{thm:tt_disp} are also explicit. Here we mention two special cases: 

\subsubsection*{Intervals} If $\Lambda(A^{(j)}) \subseteq [a,b]$ for $0 < a < b < \infty$, then one can take $E_j = [ja,jb]$ and $F_j=[-(d-j)b,-(d-j)a]$ in~\cref{thm:tt_disp}.  From~\cite[Cor.~4.2]{beckermann2017singular}, we find that 
\[
Z_{k_j}(E_j,F_j) \leq 4\left[\exp \left(\frac{\pi^2}{2\log{(16\gamma_j)}}\right)\right]^{-2k}, \quad \gamma_j = \frac{(da+j(b-a))(db-j(b-a))}{abd^2}. 
\]
In particular, the following bound holds: 
\begin{equation} \label{eq:tt_real}
\red{\ttsto_\epsilon(\mathcal{X}) \le \sum_{j = 1}^d (k_{d-1}\nu_{d-1})(k_d\nu_d)n_d}, \qquad k_j=\Bigg\lceil \!\frac{\log(16\gamma_j)\log(4\sqrt{d}/\epsilon)}{\pi^2} \Bigg\rceil,
\end{equation}
where $\nu_j=\rank(G_j)$.

\subsubsection*{Disks} If $\Lambda(A^{(j)}) \subseteq \{z \in \C \ : \ |z-z_0| \le \eta\}$ for $0 < \eta < z_0$ and $z_0, \eta \in \R$, then one finds that $\Lambda(A^{(1)})+ \dots +\Lambda(A^{(j)}) \subseteq \{z \in \C \ : \ |z-jz_0| \le j\eta\}$ and $-(\Lambda(A^{(j+1)})+ \dots +\Lambda(A^{(d)})) \subseteq \{z \in \C \ : \ |z+(d-j)z_0| \le (d-j)\eta\}$. From~\cite[p.~123]{starke1992near}, we find that 
\[
Z_{k_j}(E_j,F_j) = \rho_j^{-k_j}, \qquad \rho_j = \frac{2j(d-j)\eta^2}{d^2z_0^2 - ((d-j)^2+j^2)\eta^2-\sqrt{\xi_j}}, 
\]
where $\xi_j = \left(d^2z_0^2- ((d-j)^2+j^2)\eta^2\right)^2 -4j^2(d-j)^2\eta^4$. In particular, 
\begin{equation} \label{eq:tt_disk}
\red{\ttsto_\epsilon(\mathcal{X}) \le \sum_{j = 1}^d (k_{d-1}\nu_{d-1})(k_d\nu_d)n_d}, \qquad k_j = \Bigl\lceil \log(\sqrt{d}/\epsilon)/\log(\rho_j)\Bigr\rceil,
\end{equation}
where $\nu_j=\rank(G_j)$.

In~\cref{sec:examples}, we use~\cref{eq:tt_real} to bound the numerical storage cost in tensor-train format of the Hilbert tensor and the solution tensor of a discretized Poisson equation.

\subsection{The compressibility of tensors with displacement structure in the Tucker format}\label{sec:Tucker}
The matrix SVD can be used to calculate the numerical multilinear rank~\cite{kolda2009tensor,de2000multilinear}. Indeed, if $\mathcal{X} \in \C^{n_1 \times \dots \times n_d}$ is a tensor and $0 < \epsilon < 1$, then there exists a tensor $\tilde{\mathcal{X}}$ such that~\cite{de2000multilinear}:
\[
||\mathcal{X}-\tilde{\mathcal{X}}||_F \le \epsilon ||\mathcal{X}||_F, \quad \mlrank(\tilde{\mathcal{X}}) =(\rank_\delta(X_{(1)}),\dots,\rank_\delta(X_{(d)})),
\]
where $\delta=\epsilon/\sqrt{d}$ and $X_{(j)}$ is the $j$th matricization of $\mathcal{X}$.

Since the first unfolding of $\mathcal{X}$ coincides with the first matricization of $\mathcal{X}$, the bound on \red{the second element of the tensor-train rank is also a bound on the first element of the} multilinear rank of $\mathcal{X}$. One can use a similar idea to bound all entries of the multilinear ranks by considering the various matricizations. However, one finds that the spectra of $A^{(1)},\dots,A^{(d)}$ need to be separated in a slightly different sense. 

\begin{definition} \label{def:minkowski_separated_singly}
We say that normal matrices $A_1,\dots,A_d$ are Minkowski singly separated if there are disjoint sets $E_j$ and $F_j$ so that
\[ \Lambda(A_j) \subseteq E_j, \quad -\!\left(\!\sum_{k=1,k \neq j}^d\! \Lambda(A_k)\!\right) \subseteq F_j, \quad 1 \le j \le d, \]
where the set additions are Minkowski sums and $\Lambda(A_j)$ denotes the spectrum of $A_j$.
\end{definition}

\Cref{fig:min_sing} illustrates the spectra of Minkowski singly separated matrices $A^{(1)}$, $A^{(2)}$, and $A^{(3)}$ along with their enclosed sets and Minkowski sums of the sets. Under this separation condition, we have the following theorem:

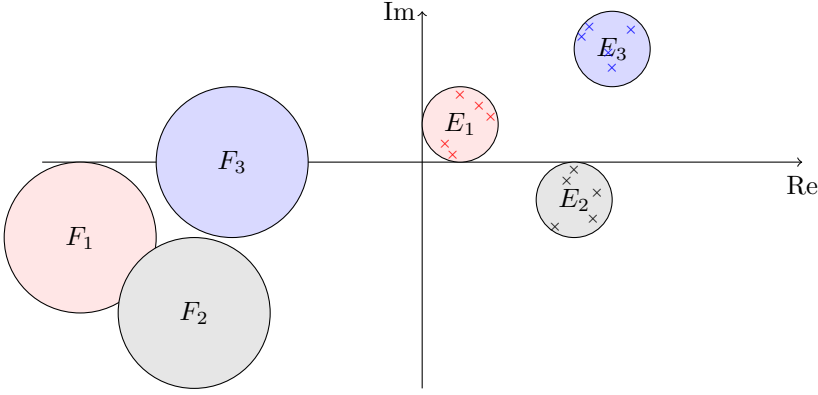
\begin{figure}
\label{fig:min_sing}
\centering
\begin{tikzpicture}
\draw[->] (-5,0)--(5,0);
\filldraw[black] (5,-0.3) node {Re};
\draw[->] (0,-3)--(0,2);
\filldraw[black] (-0.3,2) node {Im};
\filldraw[color=black,fill=pink!40] (0.5,0.5) circle (0.5);
\filldraw[black] (0.5,0.5) node {$E_1$};
\filldraw[red] (0.75,0.75) node {\tiny{$\times$}};
\filldraw[red] (0.9,0.6) node {\tiny{$\times$}};
\filldraw[red] (0.3,0.25) node {\tiny{$\times$}};
\filldraw[red] (0.4,0.1) node{\tiny{$\times$}};
\filldraw[red] (0.5,0.9) node{\tiny{$\times$}};
\filldraw[color=black,fill=gray!20] (2,-0.5) circle (0.5);
\filldraw[black] (2,-0.5) node {$E_2$};
\filldraw[black] (2.25,-.75) node {\tiny{$\times$}};
\filldraw[black] (1.75,-0.85) node {\tiny{$\times$}};
\filldraw[black] (2.3,-0.4) node {\tiny{$\times$}};
\filldraw[black] (2,-0.1) node {\tiny{$\times$}};
\filldraw[black] (1.9,-0.25) node {\tiny{$\times$}};
\filldraw[color=black,fill=blue!15] (2.5,1.5) circle (0.5);
\filldraw[black] (2.5,1.5) node {$E_3$};
\filldraw[blue] (2.75,1.75) node {\tiny{$\times$}};
\filldraw[blue] (2.1,1.66) node {\tiny{$\times$}};
\filldraw[blue] (2.5,1.25) node {\tiny{$\times$}};
\filldraw[blue] (2.2,1.8) node {\tiny{$\times$}};
\filldraw[blue] (2.45,1.45) node {\tiny{$\times$}};
\filldraw[color=black,fill=pink!40] (-4.5,-1) circle (1);
\filldraw[black] (-4.5,-1) node {$F_1$};
\filldraw[color=black,fill=blue!15] (-2.5,0) circle (1);
\filldraw[black] (-2.5,-0) node {$F_3$};
\filldraw[color=black,fill=gray!20] (-3,-2) circle (1);
\filldraw[black] (-3,-2) node {$F_2$};
\end{tikzpicture}
\caption{Minkowski singly separated matrices $A^{(1)},A^{(2)}$, and $A^{(3)}$ where the colored crosses denote the spectrum of $A^{(1)}$, $A^{(2)}$, and $A^{(3)}$, respectively. Here, $\Lambda(A^{(1)}) \subseteq E_1$, $\Lambda(A^{(2)}) \subseteq E_2$, $\Lambda(A^{(3)}) \subseteq E_3$, $-(\Lambda(A^{(1)})+\Lambda(A^{(2)})) \subseteq F_3$, $-(\Lambda(A^{(1)})+\Lambda(A^{(3)})) \subseteq F_2$, and $-(\Lambda(A^{(2)})+\Lambda(A^{(3)})) \subseteq F_1$. By definition, we have that $E_1$ is disjoint from $F_1$ (red regions), that $E_2$ is disjoint from $F_2$ (gray regions), and that $E_3$ is disjoint from $F_3$ (blue regions).}
\end{figure}

\begin{theorem} \label{thm:ml_disp}
Suppose $\mathcal{X} \in \C^{n_1 \times \dots \times n_d}$ satisfies~\cref{eq:tensor_disp}, where $A^{(1)},\dots,A^{(d)}$ are Minkowski singly separated with disjoint sets $E_j$ and $F_j$ for $1 \le j \le d$. Then, for a fixed $0 < \epsilon < 1$, we have
\[
\red{\mlsto_\epsilon(\mathcal{X}) \le \sum_{j=1}^d n_jk_j\mu_j+\prod_{j=1}^d k_j\mu_j}, \quad \quad \mlrank(\mathcal{G})=(\mu_1, \dots, \mu_d),
\]
where $k_j$ is an integer so that $Z_{k_j}(E_j,F_j) \le \epsilon/\sqrt{d}$.
\end{theorem}
\begin{proof}
\red{One can bound all the entries of the multilinear rank vector of $\mathcal{X}$ by the second entry of} the tensor-train rank vector of the tensors $\mathcal{Y}^1,\dots,\mathcal{Y}^d$ (see~\cref{eq:cyc_mat}). Due to the way $\mathcal{Y}^j$ is constructed, it can be shown that $\mathcal{Y}^j$ satisfies
\[ \mathcal{Y}^j \times_1 A^{(j)} + \cdots + \mathcal{Y}^j \times_{d-j+1} A^{(d)} + \mathcal{Y}^j \times_{d-j+2} A^{(1)}+ \cdots + \mathcal{Y}^j \times_d A^{(j-1)} = \mathcal{H}^j, \]
where $H^j_{(1)}=G_{(j)},\dots,H^j_{(d-j+1)}=G_{(d)},H^j_{(d-j+2)}=G_{(1)},\dots,H^j_{(d)}=G_{(j-1)}$ and $\mathcal{H}^j$ is constructed from $\mathcal{G}$ in the same way that $\mathcal{Y}^j$ is constructed from $\mathcal{X}$. The result follows from~\Cref{thm:tt_disp} as the $j$th element of the multilinear rank of $\mathcal{X}$ is bounded above by the bound of the \red{second entry} of the tensor-train rank of $\mathcal{Y}^j$.
\end{proof}

As before, explicit \red{bounds on the compressibility in Tucker format} can be obtained from~\cref{thm:ml_disp} by special choices of $E_j$ and $F_j$ such as when they are intervals or disks.

\subsubsection*{Intervals} If $\Lambda(A^{(j)}) \subseteq [a,b]$ for $0 < a < b < \infty$, then one can take $E_j = [a,b]$ and $F_j = [-(d-1)b,-(d-1)a]$. Therefore, we find that~\cite[Cor.~4.2]{beckermann2017singular}
\[
\red{\mlsto_\epsilon(\mathcal{X}) \le k\sum_{j=1}^d n_j\mu_j+k^d\prod_{j=1}^d \mu_j}, \qquad k=\Bigg\lceil \!\frac{\log(16\gamma)\log(4\sqrt{d}/\epsilon)}{\pi^2} \Bigg\rceil, 
\]
where $\gamma = (da+(b-a))(db-(b-a))/(abd^2)$ and $\mlrank(\mathcal{G})=(\mu_1,\dots,\mu_d)$.

\subsubsection*{Disks} If $\Lambda(A^{(j)}) \subseteq \{z \in \C \ : \ |z-z_0| \le \eta\}$ for $0 < \eta < z_0$ and $z_0, \eta \in \R$, then one can take $E_j = \{z \in \C \ : \ |z-z_0| \le \eta\}$ and $F_j = \{z \in \C \ : \ |z+(d-1)z_0| \le (d-1)\eta\}$.  From~\cite[p.~123]{starke1992near}, we find that 
\[
\red{\mlsto_\epsilon(\mathcal{X}) \le k\sum_{j=1}^d n_j\mu_j+k^d\prod_{j=1}^d \mu_j}, \qquad k = \Bigl\lceil \log(\sqrt{d}/\epsilon)/\log(\rho)\Bigr\rceil,
\]
where $\rho = (2(d-1)\eta^2)/(d^2z_0^2 - ((d-1)^2+1)\eta^2-\sqrt{\xi})$, $\xi = (d^2z_0^2- ((d-1)^2+1)\eta^2)^2 -4(d-1)^2\eta^4$, and $\mlrank(\mathcal{G})=(\mu_1,\dots,\mu_d)$.

\subsection{The compressibility of tensors with displacement structure in the CP format}\label{sec:CP} 
While deriving the bounds in this paper, we also considered including bounds on the \red{compressibility of tensors with displacement structure in CP format}. We were unable to come up with nontrivial bounds unless we introduced several additional and arbitrary assumptions in the statements of our theorem.  

\subsection{Worked examples} \label{sec:examples}
Here, we give two examples that illustrate how to use the displacement structure of a tensor to understand its compressibility. Since the bounds in \red{tensor-train format and Tucker format are related through ranks}, we only show results \red{for the tensor-train format. As in the previous examples, we use the second element of the tensor-train rank and its bound to visualize the compressibility}. We consider two tensors: (1) The 3D Hilbert tensor and (2) The solution tensor of a Poisson equation. 

\subsubsection{The 3D Hilbert tensor}
Consider the Hilbert tensor $\mathcal{H} \in \C^{n \times n \times n}$ defined by
\[ \mathcal{H}_{ijk} = \frac{1}{i+j+k-2}, \qquad 1 \le i,j,k \le n. \]
This tensor is analogous to the notoriously ill-conditioned Hilbert matrix~\cite{hilbert1894beitrag,eckart1936approximation}. It is easy to verify that the tensor possesses the following displacement structure:
\[ 
\mathcal{H} \times_1 D + \mathcal{H} \times_2 D + \mathcal{H} \times_3 D = \mathcal{S}, 
\]
where $\mathcal{S}$ is the tensor of all ones and $D$ is a diagonal matrix with $D_{ii} = i-\frac{2}{3}$. Thus, $\rank(\mathcal{S})=1$ and the ranks of the unfoldings of $\mathcal{S}$ are all 1. 

Since the spectrum of $D$ is contained in $[\frac{1}{3}, \frac{3n-2}{3}]$,~\cref{eq:tt_real} tells us that for any $0 < \epsilon < 1$ we have
\begin{equation} \label{hilbert_bound}
\red{\ttsto_{\epsilon}(H) \le n(s_1^2+2s_1)}, \qquad s_1=\Biggl\lceil \frac{1}{\pi^2}\log\!\left(\frac{16n(2n-1)}{3n-2}\right)\log\!{\left(\frac{4\sqrt{3}}{\epsilon}\right)} \Biggr\rceil.
\end{equation}
That is, $s_1 = \mathcal{O}(\log n \log(1/\epsilon))$ and means that the $n\times n\times n$ Hilbert tensor can be stored, up to an accuracy of $\epsilon$ in the Frobenius norm, in just $\mathcal{O}(n (\log n)^2(\log(1/\epsilon))^2)$ degrees of freedom. 
\Cref{fig:DispHilbert} (left) shows \red{the compressibility} of $\mathcal{H}$ \red{with $n = 100$ by computing the ratio of the storage costs using tensor-train format and explicit storage. Our theoretical results bound the savings well. \Cref{fig:DispHilbert} (right) shows the compressibility of $\mathcal{H}$ by plotting $s_1$ and its bound in~\eqref{hilbert_bound} for different values of $n$}. The actual tensor-train ranks of $\mathcal{H}$ are computed with the TT-SVD algorithm~\cite{oseledets2011tensor}.

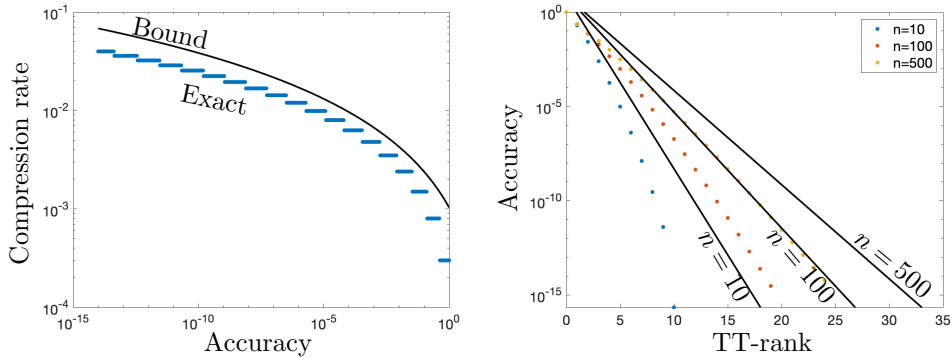
\begin{figure} 
\begin{minipage}{.49\textwidth}
\begin{overpic}[width=\textwidth]{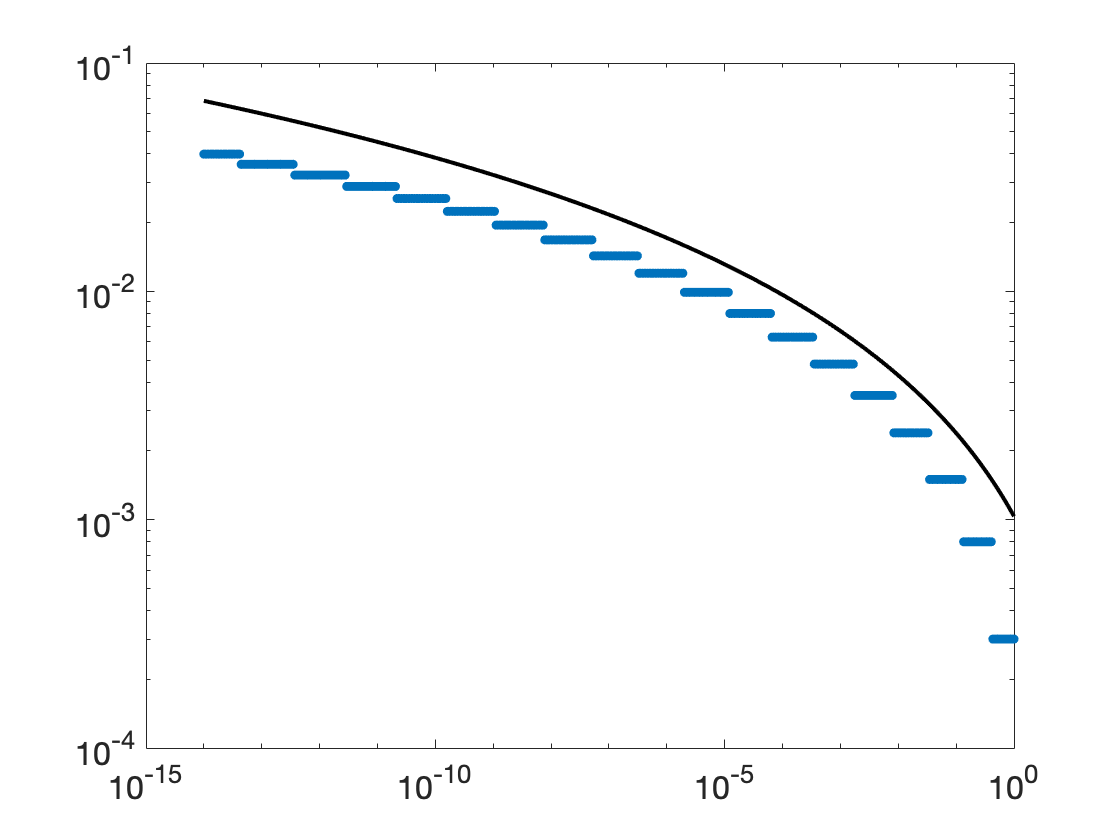}
\put(40,-1) {Accuracy}
\put(0,17) {\rotatebox{90}{Compression rate}}
\put(35,51) {\rotatebox{-12}{Exact}}
\put(25,65) {\rotatebox{-12}{Bound}}
\end{overpic} 
\end{minipage} 
\begin{minipage}{.49\textwidth}
\begin{overpic}[width=\textwidth]{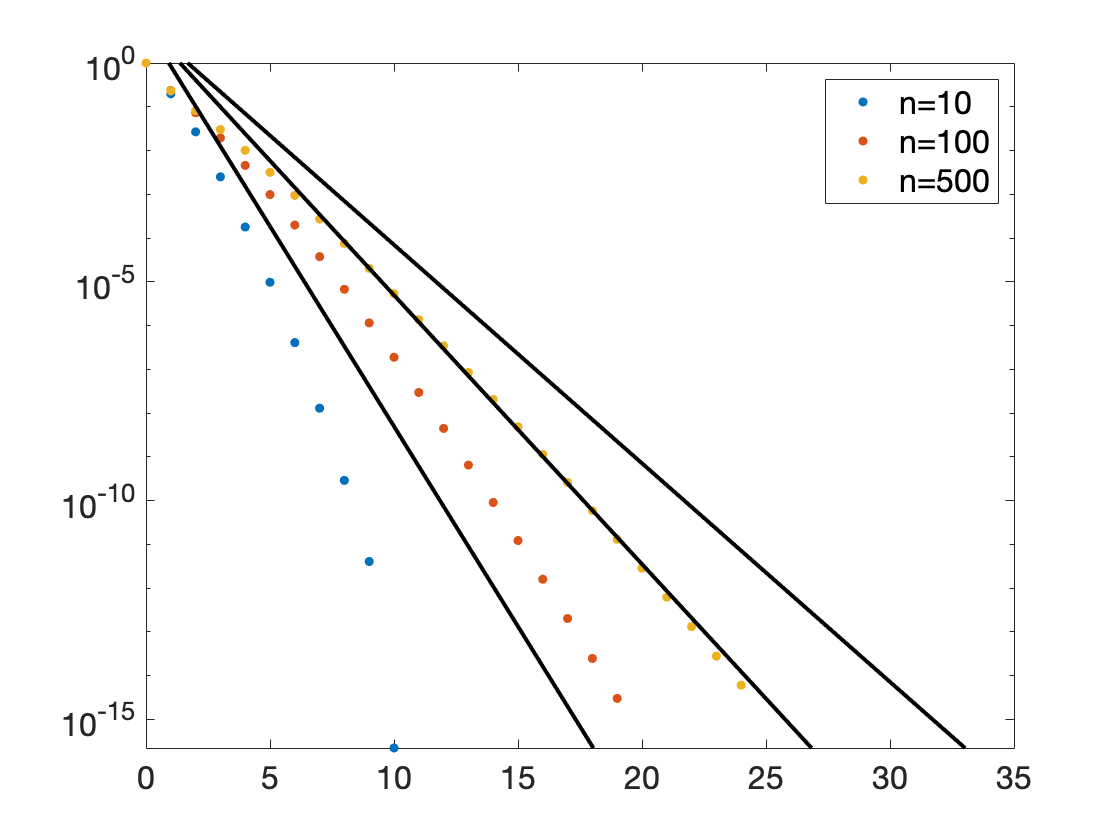}
\put(43,-1) {TT-rank}
\put(-1,28) {\rotatebox{90}{Accuracy}}
\put(72,22) {\rotatebox{-40}{$n=500$}}
\put(54,23) {\rotatebox{-50}{$n=100$}}
\put(40,23) {\rotatebox{-60}{$n=10$}}
\end{overpic} 
\end{minipage} 
\caption{\red{The compressibility of the 3D Hilbert tensor in tensor-train format}. Left: \red{The ratio between the storage cost for representing a $100 \times 100 \times 100$ Hilbert tensor in a tensor-train format and $100^3$ (blue dots), along with our theoretical bound on the compression rate (black line)}. Right: The \red{second element of the} tensor-train rank, $s_1$, (dots) and the theoretical bound (black lines) for $n = 10,100,$ and $500$.} 
\label{fig:DispHilbert} 
\end{figure}

\subsubsection{Tensor solution of a discretized Poisson equation}\label{sec:PoissonExample}
\red{Tensor decompositions can be incorporated into efficient solvers of partial differential equations~\cite{beylkin2005algorithms, oseledets2012solution, ballani2013projection, sun2014numerical, khoromskij2017rank, khoromskij2018tensor}.  Displacement structure arises for the solution tensor when one discretizes a Laplace operator, or any Laplace-like operator. Here, }consider the 3D Poisson equation on $[-1,1]^3$ with zero Dirichlet conditions, i.e., 
\begin{equation} \label{eq:poiss_eqn}
-(u_{xx}+u_{yy}+u_{zz})=f\ {\rm on} \ \Omega=[-1,1]^3, \qquad u|_{\partial \Omega}=0.
\end{equation}
If one writes down a second-order finite difference discretization of~\cref{eq:poiss_eqn} on an $n \times n \times n$ equispaced grid, then one obtains the following multidimensional Sylvester equation
\[
\mathcal{X}\times_1 K + \mathcal{X}\times_2 K + \mathcal{X}\times_3 K = \mathcal{F}, \qquad K = -\frac{1}{h^2}\begin{bmatrix}2 & -1 \cr -1 & \ddots & \ddots \cr & \ddots & \ddots & -1 \cr & & -1 & 2 \end{bmatrix},
\]
where $h = 2/n$ and $\mathcal{F}_{ijk} = f(i h-1,j h-1,k h-1)$ for $1\leq i,j,k\leq n-1$. The solution tensor $\mathcal{X}$ is unknown and for large $n$, one hope that $\mathcal{X}_{ijk}\approx u(i h-1,j h-1,k h-1)$ for $1\leq i, j, k\leq n-1$ is a reasonably good approximation.  The eigenvalues of $K$ are given by $4/h^2\sin^2(\pi k/(2n))$ for $1 \leq k \leq n$ with $h = 2/n$~\cite[(2.23)]{leveque2007finite}. Since $(2/\pi)x \leq \sin x \leq 1$ for $x \in [0,\pi/2]$ and $h = 2/n$, the eigenvalues of $K$ are contained in the interval $[1,n^2]$.

We are interested in understanding the \red{compressibility of $\mathcal{X}$ in tensor-train format} when $f = 1$.  Since $\Lambda(K)\subseteq [1,n^2]$ and $\ttrank(\mathcal{F}) = (1,1,1,1)$,~\cref{eq:tt_real} gives
\begin{equation} \label{eq:poisson_tt_bound}
\red{\ttsto_{\epsilon}(\mathcal{X}) \le n(s_1^2+2s_1)}, \quad s_1=\Biggl\lceil \frac{1}{\pi^2}\log\left(\frac{16(n^2+2)(2n^2+1)}{9n^2}\right)\log\left(\frac{4\sqrt{3}}{\epsilon}\right) \Biggr\rceil.
\end{equation}

\Cref{fig:Poisson_tt} (left) shows the \red{second element of the} tensor-train rank, $s_1$, and the bound of the approximate solution tensor to the Poisson equation via finite difference discretization.
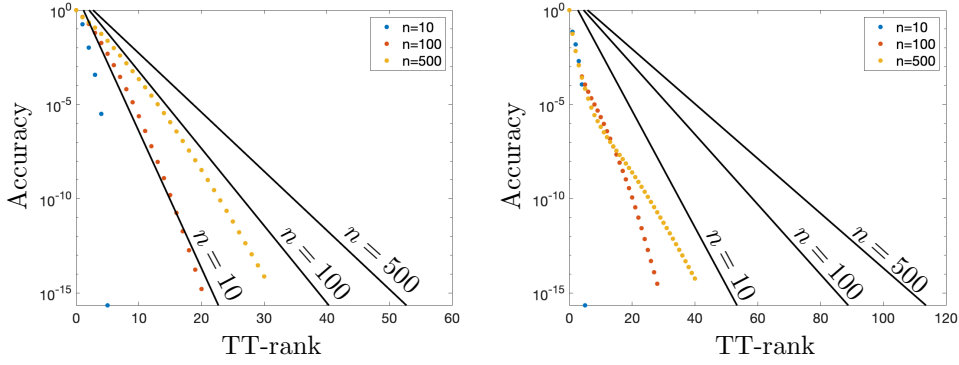
\begin{figure} 
\centering
\begin{minipage}{0.49\textwidth}
\begin{overpic}[width=\textwidth]{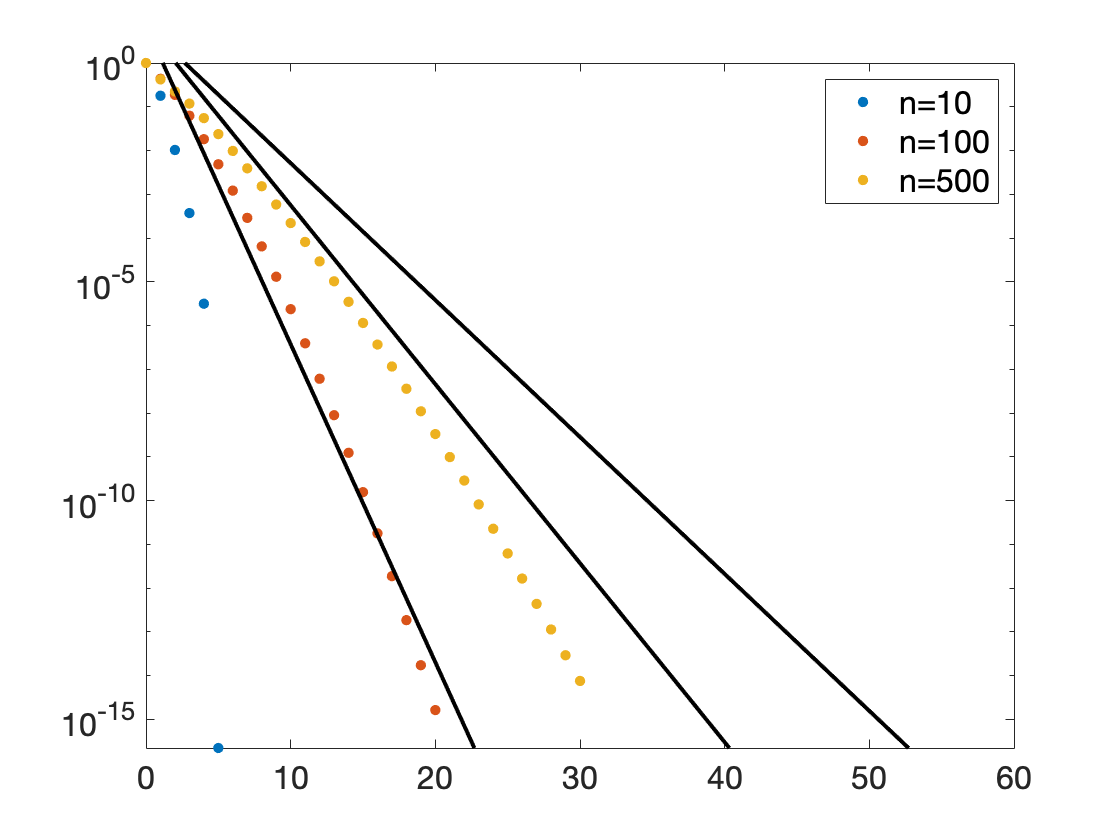}
\put(43,-2) {TT-rank}
\put(-1,28) {\rotatebox{90}{Accuracy}}
\put(37,23) {\rotatebox{-62}{$n=10$}}
\put(55,23) {\rotatebox{-49}{$n=100$}}
\put(68,23) {\rotatebox{-42}{$n=500$}}
\end{overpic} 
\end{minipage}
\begin{minipage}{0.49\textwidth}
\begin{overpic}[width=\textwidth]{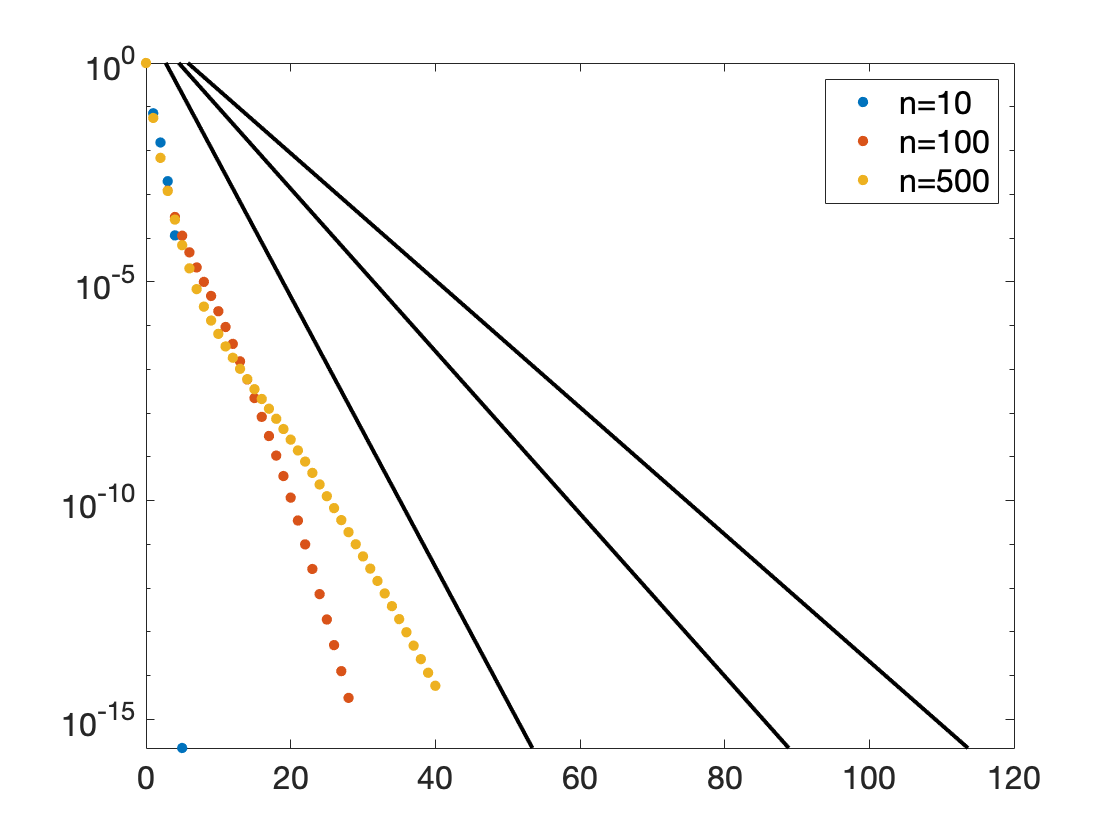}
\put(43,-2) {TT-rank}
\put(-1,28) {\rotatebox{90}{Accuracy}}
\put(41,23) {\rotatebox{-62}{$n=10$}}
\put(59,23) {\rotatebox{-49}{$n=100$}}
\put(71,23) {\rotatebox{-42}{$n=500$}}
\end{overpic} 
\end{minipage}
\caption{Left: The \red{second element of the} tensor-train rank, $s_1$, (blue, red, and yellow dots) of the finite difference solution to $-(u_{xx}+u_{yy}+u_{zz})=1$ on $[-1,1]^3$ with zero Dirichlet conditions, and the theoretical bound in~\cref{eq:poisson_tt_bound} (black lines). \red{Right: The \red{second element of the} tensor-train rank, $s_1$, (blue, red, yellow dots) of the ultraspherical spectral solution to $-(u_{xx}+u_{yy}+u_{zz})=1$ on $[-1,1]^3$ with zero Dirichlet conditions, and the theoretical bound in~\cref{eq:poisson_tt_bound_sp} (black lines).}}
\label{fig:Poisson_tt} 
\end{figure}

\red{One wonders if there is also a fast Poisson solver for spectral discretizations. This turns out to be feasible with a carefully constructed ultraspherical spectral discretization. The Poisson equation can be discretized to a tensor equation as~\cite{fortunato2017fast}:
\begin{equation} \label{eq:pois_sp_sylv}
\mathcal{X} \times_1 A^{-1} + \mathcal{X} \times_2 A^{-1} + \mathcal{X} \times_3 A^{-1} = \mathcal{G}, 
\end{equation}
where 
\[ u(x,y,z) = (1-x^2)(1-y^2)(1-z^2) \sum_{p=0}^n \sum_{q=0}^n \sum_{r=0}^n \mathcal{X}_{pqr} \tilde{C}^{(3/2)}_p(x) \tilde{C}^{(3/2)}_q(y) \tilde{C}^{(3/2)}_r(z), \]
\[f(x,y,z) = \sum_{p=0}^n \sum_{q=0}^n \sum_{r=0}^n \mathcal{F}_{pqr} \tilde{C}^{(3/2)}_p(x) \tilde{C}^{(3/2)}_q(y) \tilde{C}^{(3/2)}_r(z), \] 
$\tilde{C}_k^{(3/2)}$ is the degree $k$ orthonormalized ultraspherical polynomial with parameter $\frac{3}{2}$~\cite[Table 18.3.1]{olver2010nist}, $\mathcal{G} = \mathcal{F} \times_1 M^{-1} \times_2 M^{-1} \times_3 M^{-1}$, $A = D^{-1}M$, $D$ is a diagonal matrix, $M$ and $A$ are both symmetric pentadiagonal matrices, and the spectrum of $A$ satisfies $\Lambda(A) \in [-1,\ -1/(30n^4)]$. If $f = 1$, ~\cref{eq:tt_real} gives
\begin{equation} \label{eq:poisson_tt_bound_sp}
\ttsto_{\epsilon}(\mathcal{X}) \le n(s_1^2+2s_1), \ s_1=\Biggl\lceil \frac{1}{\pi^2}\log\left(\frac{16(30n^4+2)(60n^4+1)}{270n^4}\right)\log\left(\frac{4\sqrt{3}}{\epsilon}\right) \Biggr\rceil.
\end{equation}}

\red{\Cref{fig:Poisson_tt} (right) shows the second element of the tensor-train rank, $s_1$, and the bound of the approximate solution tensor to the Poisson equation via ultraspherical spectral discretization. This spectral discretization indicates that the $n \times n \times n$ tensor discretization of the solution can be approximated with only $\mathcal{O}(dn(\log n)^2(\log (1/\epsilon))^2)$ degrees of freedom. This is a significant reduction in the cost of storing the solution, with a relatively straightforward decomposition. Comparatively, one can achieve $\mathcal{O}(d \log n \log(1/\epsilon))$ with quantics tensor formats~\cite{khoromskij2011dlog, kazeev2018quantized}, but those require more complicated representations.}


\red{Some special functions can be well-approximated by exponential sums of the form
\[ S_k(x) = \sum_{j=1}^k \alpha_j e^{-t_jx}, \qquad \alpha_j,t_j \in \R, \]
and these approximant can be used to represent the solution to PDEs with Laplace-like operators~\cite{grasedyck2004existence,khoromskij2009tensor}. In~\cite{khoromskij2009tensor}, the author uses exponential sums to show that the solution tensor to several 3D elliptic PDEs can be $\mathcal{O}(dn(\log n)^2(\log (1/\epsilon))^2)$ degrees of freedom. However, the constants in these compressibility statements are left implicit. 
In general, both exponential sum approximation and Zolotarev numbers can be used to bound the $k$th singular value of matrices with displacement structure and capture the geometric decay, but the Zolotarev bound tends to be tighter and does not involve an algebraic factor related to $k$~\cite{townsend2017singular}.}

\subsection{Solving for tensors in compressed formats} \label{sec:poisson_solver}
\red{Since the proof of~\Cref{thm:tt_disp} and~\Cref{thm:ml_disp} are constructive, we can use their implicit algorithms to solve 3D tensor Sylvester equations of the form:
\begin{equation} \label{eq:3d_syl}
\mathcal{X} \times_1 A^{(1)} + \mathcal{X} \times_2 A^{(2)} + \mathcal{X} \times_3 A^{(3)} = \mathcal{F},
\end{equation}
where $A^{(1)} \in \C^{n_1 \times n_1}$, $A^{(2)} \in \C^{n_2 \times n_2}$, $A^{(3)} \in \C^{n_3 \times n_3}$, and $\mathcal{F} \in \C^{n_1 \times n_2 \times n_3}$.  In particular, we can compute approximate solutions to~\cref{eq:3d_syl} in tensor-train or Tucker format when $\mathcal{F}$ is a low rank tensor and the spectra of $A^{(1)}$, $A^{(2)}$, and $A^{(3)}$ are well-separated. If $A^{(1)}$, $A^{(2)}$, and $A^{(3)}$ are Minkowski sum separated, and the unfoldings $F_1$ and $F_2$ of $\mathcal{F}$ have low rank decompositions $F_1 = W_1Z_1^*$, and $F_2 = W_2Z_2^*$ with rank $r_1$ and $r_2$, respectively, then we can solve for $\mathcal{X}$ in tensor-train format.}

\red{The tensor-train factors of $\mathcal{X}$ obtained by the TT-SVD algorithm are orthogonal matrices for the column and row spaces of unfoldings of $X$.
For example, the first tensor-train factor $U_1$ of $\mathcal{X}$ can be found as a matrix with orthonormal columns spanning the column space of the first unfolding $X_1$. Since $\mathcal{X}$ satisfies~\cref{eq:3d_syl}, we find that $X_1$ satisfies the Sylvester equation
\begin{equation} 
A^{(1)}X_1 + X_1(I \otimes A^{(2)}+A^{(3)} \otimes I)^T = W_1Z_1^*. 
\label{eq:X1} 
\end{equation} 
We can use the factored alternating direction implicit (fADI) method to solve~\cref{eq:X1} for a matrix $V_1$ such that $X_1 = V_1D_1Y_1^*$~\cite{benner2009adi}. One can then use the QR decomposition of $V_1$, i.e., $V_1 = U_1R_1$, to calculate the first tensor-train core $U_1$.}

\red{Second and third tensor-train factors can be computed by finding matrices with orthonormal columns for the column and row spaces associated to $C_2$, where $C_2={\rm reshape}(R_1D_1Y_1^*, s_1n_2, n_3)$. It can be shown that $C_2$ satisfies the Sylvester equation
\[
(I \otimes (U_1^*A^{(1)}U_1)+A^{(2)} \otimes I)C_2+C_2 (A^{(3)})^T = (I \otimes U_1^*)W_2Z_2^*. 
\]
One can, again, use fADI to solve for a low rank decomposition of $C_2$, i.e., $C_2=V_2D_2Y_2^*$. This low rank decomposition can be compressed by performing a QR factorization of $V_2$ and $Y_2$ and then doing a SVD to obtain $C_2 \approx U_2\Sigma T_2^*$, where $U_2$ and $T_2$ are matrices with $s_2$ orthonormal columns and $\Sigma$ is a diagonal matrix. In this way, the second tensor-train factor is $U_2={\rm reshape}(U_2, [s_1,n_2,s_2])$ and the third factor $U_3=\Sigma T_2^*$. Although the fADI method requires the solution of shifted linear systems with $I \otimes (U_1^*A^{(1)}U_1)+A^{(2)} \otimes I$, the Kronecker product structure allows one to reshape these linear systems into Sylvester equations, which can themselves be solved with the alternating direction implicit (ADI) method~\cite{benner2009adi}. That is, one can completely avoid solving a huge linear system. As a result, if $n_1=n_2=n_3=n$, and $A^{(1)}$, $A^{(2)}$, and $A^{(3)}$ have structures so that shifted linear systems can be solved in $\mathcal{O}(n)$, then the solver has a complexity that is less than $\mathcal{O}(n^3)$. In summary, the ADI-based tensor Sylvester equation solver is the following:}

\begin{algorithm} 
\caption{A 3D Sylvester equation~\eqref{eq:3d_syl} solver that solves the solution in tensor-train form.}
\begin{algorithmic}[1]
\State Use f-ADI to solve for the column space $Z_1$ of $X_1$ that satisfies $A^{(1)}X_1 + X_1(I \otimes A^{(2)}+A^{(3)} \otimes I)^T = F_1 = M_1N_1^*$.
\State Perform a $QR$ decomposition, $Z_1=U_1R_1$, and let $U_1 = U_1(:,1:s_1)$ if $R_1(s_1+1,s_1+1)$ is small enough.
\State Use f-ADI to solve for $C_2=Z_2D_2Y_2^*$ where $C_2$ satisfies $(I \otimes (U_1^*A^{(1)}U_1)+A^{(2)} \otimes I)C_2+C_2 (A^{(3)})^T = (I \otimes U_1^*)F_2 = (I \otimes U_1^*)M_2N_2^*$.
\State Find a low rank decomposition of $C_2 \approx U_2\Sigma T_2^*$ using $Z_2, D_2$ and $Y_2$, and denote the rank by $s_2$.
\State Let $U_2={\rm reshape}(U_2, [s_1,n_2,s_2])$.
\State Let $U_3=\Sigma T_2^*$.
\State The solution $\mathcal{X}$ is in the tensor-train form with cores $U_1$, $U_2$, and $U_3$.
\end{algorithmic}   
\end{algorithm}

\red{Similarly, if all matricizations of $\mathcal{F}$ are low rank, and $A^{(1)}$, $A^{(2)}$, and $A^{(3)}$ are Minkowski singly separated, then we can solve for the solution in orthogonal Tucker format via the higher order singular value decomposition (HOSVD) method~\cite{de2000multilinear}. Each factor matrix of $\mathcal{X}$ is a matrix with orthonormal columns that span the column space of the matricization of $\mathcal{X}$, which satisfies the Sylvester equation:
\[ A^{(j)}X_{(j)} + X_{(j)}(I \otimes A^{(i)}+A^{(k)} \otimes I)^T = F_{(j)}, \]
where
\[ i = \begin{cases} 
      1, & j = 3, \\
      j+1, & j = 1,2,
   \end{cases} \qquad
   k = \begin{cases} 
      3, & j = 1, \\
      j-1, & j = 2,3.
   \end{cases}
\]
If solving shifted linear systems with $A^{(1)}$, $A^{(2)}$, and $A^{(3)}$ is fast, then we can use fADI to solve for the orthogonal column space of $X_{(j)}$, and use a direct method, such as a 3D Bartels--Stewart algorithm to solve for the core tensor~\cite{bartels1972solution}.}

\subsubsection{Poisson equation solver}
\red{Consider the example of Poisson equation in~\cref{sec:PoissonExample} with ultraspherical discretization~\cref{eq:pois_sp_sylv}. Since $A$ is a penta-diagonal matrix, we can solve shifted linear systems with $A^{-1}$ in $\mathcal{O}(n)$ time. Therefore, we can obtain a fast Poisson equation solver that computes the solution in tensor-train or orthogonal Tucker format. The complexity for the tensor-train format solver is $\mathcal{O}(n (\log n)^3 (\log (1/\epsilon))^3)$, where $0<\epsilon<1$ is the accuracy.}

\red{\Cref{fig:Poisson_ex} shows the running time of different discretized Poisson solvers. The red line represents the direct solver that converts~\eqref{eq:poisson_tt_bound_sp} into a huge linear system via Kronecker product. The green line represents an eigen decomposition solver, which computes the eigen decomposition of $A$ to diagonalize the equation, and solves each element of $\mathcal{X}$ directly by scaling. The blue line represents our fADI-based tensor-train solver. We can see as $n$ gets large, our algorithm is the winner.\footnote{The fADI solver is implemented in C++, while the direct and the eigen solvers are implemented in MATLAB. However, both backslash linear system solver and eigen decomposition are carried out in LAPACK, so our comparison of the three solvers is still fair. All timings are performed in MATLAB R2019a on the super computer of Cornell's Math department.}}

\begin{figure}
\centering
\begin{minipage}{0.49\textwidth}
\begin{overpic}[width=\textwidth]{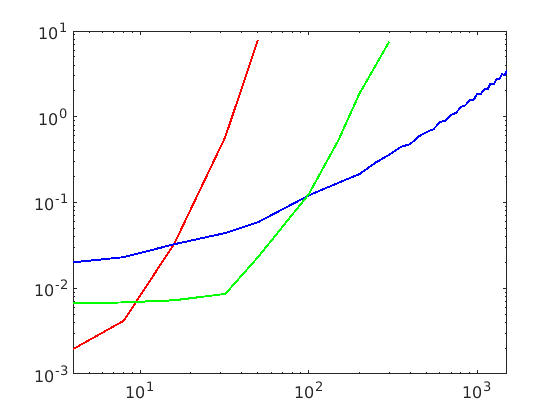}
\put(48,-2) {Size, $n$}
\put(-1,28) {\rotatebox{90}{Time (sec)}}
\put(32,42) {\rotatebox{65}{Direct}}
\put(51,44) {\rotatebox{57}{Eigen}}
\put(70,42) {\rotatebox{37}{fADI}}
\end{overpic}
\end{minipage}
\caption{\red{The execution time of a direct solver (red line), eigensolver (green line), and our fADI solver (blue line) of the spectrally discretized Poisson equation $-(u_{xx}+u_{yy}+u_{zz})=1$ on $[-1,1]^3$ with zero Dirichlet conditions with discretization size $n$ with $4\leq n\leq 1500$.}}
\label{fig:Poisson_ex} 
\end{figure}
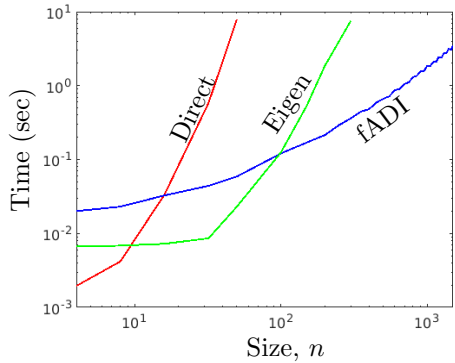

%

\section*{Acknowledgements}
We have had discussions with David Bindel, Dan Fortunato, Leslie Greengard, and Madeleine Udell about the results in this paper and appreciate their thoughts and comments. We are grateful to Nicolas Boulle and Heather Wilber for carefully reading an earlier draft of this manuscript. 

\bibliography{references}
\bibliographystyle{siam}

\end{document}